\documentclass{amsart}

\usepackage{graphicx} % Required for inserting images
\usepackage{xypic}
\usepackage{xcolor}
\usepackage[pagebackref]{hyperref}
\usepackage{amsthm}
\usepackage{amssymb}
\usepackage{comment}  % \begin{comment} \end{comment}
\usepackage[normalem]{ulem}  % \sout{} strikeout text
\usepackage{stmaryrd}

\DeclareMathOperator{\supp}{supp}

\DeclareMathOperator{\Img}{Im}

\def\pt{{\mathrm{pt}}}
\def\id{{\mathrm{id}}}

\def\ZZ{\mathbb{Z}}
\def\QQ{\mathbb{Q}}

\def\CC{\mathbb{C}}
\def\HH{\mathbb{H}}

\def\Zg{\ZZ_{\geq 0}}
\def\Zm{\Zg^m}

\def\Z{\mathcal{Z}}
\def\K{\mathcal{K}}
\def\L{\mathcal{L}}
\def\Kf{\K^{f}}
\def\ZK{\Z_\K}
\def\DJ{\mathop{\mathit{DJ}}}
\def\k{\mathbf{k}}
\def\b{\widetilde{b}}

\def\incl{\mathrm{incl}}
\def\BC{\mathop{\mathit{BC}}}
\def\FL{\mathop{\mathit{FL}}\nolimits}
\def\FQL{\mathop{\mathit{FQL}}}

\def\lbr{(\![}%{\llparenthesis}%{(\!(}
\def\rbr{]\!)
}%{\rrparenthesis}%{)\!)}
\def\X{{\underline{X}}}
\def\A{{\underline{A}}}
\def\Y{{\underline{Y}}}

\def\mdeg{\mathop{\mathrm{mdeg}}}
\def\ev{\mathop{\mathrm{ev}}\nolimits}
\def\incl{\mathop{\mathrm{incl}}\nolimits}
\def\pr{\mathop{\mathrm{pr}}\nolimits}

\newtheorem{thm}{Theorem}[section]
\newtheorem{lmm}[thm]{Lemma}

\newtheorem{prp}[thm]{Proposition}
\newtheorem{crl}[thm]{Corollary}

\theoremstyle{definition}

\newtheorem{dfn}[thm]{Definition}
\newtheorem{rmk}[thm]{Remark}

\newtheorem{question}[thm]{Question}
\newtheorem{prb}[thm]{Problem}
\newtheorem{exm}[thm]{Example}

\numberwithin{equation}{section}

\hypersetup{
    colorlinks=true,
    linkcolor=blue,
    citecolor=blue,
    urlcolor=blue,
    pdftitle={Iterated Whitehead products in DJ(K)},
    pdfauthor={Panov, Theriault, Vylegzhanin},
}

\title[Iterated Whitehead products]{Iterated Whitehead products in the homotopy groups of polyhedral products}
\author{Taras Panov}
\address{
Steklov Mathematical Institute of Russian Academy of Sciences, Moscow, Russia;
\newline
Department of Mathematics and Mechanics, Moscow
State University, Russia;\newline
Institute for Information Transmission
Problems, Russian Academy of Sciences,
Moscow, Russia;\newline
National Research University Higher School of Economics, Moscow, Russia}
\email{tpanov@mi-ras.ru}
\thanks{The work of T. Panov and F. Vylegzhanin was performed at the Steklov International Mathematical Center and supported by the Ministry of Science and Higher Education of the Russian Federation (agreement no. 075-15-2025-303).}

\author{Stephen Theriault}
\address{Mathematical Sciences, University of Southampton, Southampton, United Kingdom}
\email{S.D.Theriault@soton.ac.uk}

\author{Fedor Vylegzhanin}
\address{
Steklov Mathematical Institute of Russian Academy of Sciences, Moscow, Russia;
\newline
National Research University Higher School of Economics, Moscow, Russia}
\email{vylegf@gmail.com}



\begin{document}

\begin{abstract} 
We study structure within the homotopy groups of the Davis-Januszkiewicz space $\DJ(\K)$ associated with a simplicial complex $\K$. The inclusion of each vertex in $\K$ induces a map from the two-sphere into $\DJ(\K)$. 
These maps generate a quasi-Lie subalgebra $QL(\K)$ via the Whitehead product and a $\Pi$-subalgebra $S(\K)$ via the Whitehead product and composition. We describe the quasi-Lie subalgebra $QL(\K)$, and show that the $\Pi$-subalgebra $S(\K)$ coincides with the whole of $\pi_*(\DJ(\K))$ if and only if $\K$ is a flag complex.
Extensions to more general polyhedral products are also considered. 
\end{abstract}  

\maketitle

\section{Introduction}

Let $\K$ be a simplicial complex on the vertex set $[m]=\{1,\dots,m\}$. The corresponding \emph{Davis--Januszkiewicz space} 
\[
  \DJ(\K):=(\CC P^\infty,\ast)^\K=\bigcup_{I\in\K}(\CC P^\infty)^{\times I}
\]  
is a CW-subcomplex in $(\CC P^\infty)^m$ and a particular case of the polyhedral product construction. These spaces appear in toric topology~\cite{toric_topology} as the Borel constructions for torus actions on quasitoric manifolds and smooth toric varieties. The cohomology ring of $\DJ(\K)$ is identified with the Stanley--Reisner ring of $\K$, providing a link to combinatorial commutative algebra.

The second homotopy group $\pi_2(\DJ(\K))\cong\mathbb Z^m$ has $m$ \emph{canonical} generators represented by the maps
\[
  t_i\colon S^2 \longrightarrow \CC P^\infty \longrightarrow
  (\CC P^\infty)^{\vee m}=\DJ(\L) \longrightarrow
  \DJ(\K),\quad 1\le i\le m,
\]
where $\L$ is the simplicial complex consisting of the $m$ disjoint vertices, the left map is the inclusion of the bottom cell, the middle map is the inclusion of the $i$th wedge summand, and the right map is the map of polyhedral products induced by the simplicial inclusion $\mathcal L\to\K$.

The homotopy groups $\pi_*(\DJ(\K))$ form a quasi-Lie algebra~\cite{hilton_quasilie} with respect to the Whitehead product, and form  a $\Pi$-algebra~\cite{percy} with respect to the Whitehead product and compositions with elements of the homotopy groups of spheres. 

Rationally, the only relations satisfied by the canonical elements $t_1,\ldots,t_m\in\pi_2(\DJ(\K))\otimes\QQ$ are $[t_i,t_i]=0$ and $[t_i,t_j]=0$ for $\{i,j\}\in\K$ (see~\cite[Theorem~9.3]{pr} and also~\cite[Theorem 1.1(b)]{vylegzhanin22}). Integrally, the part of the homotopy groups $\pi_*(\DJ(\K))$ generated by the canonical elements is much more subtle, and can be studied in two formats.

\begin{prb}\label{prb:main1}
Describe the quasi-Lie subalgebra $QL(\K)$ of $\pi_*(\DJ(\K))$ generated by $t_1,\ldots,t_m\in\pi_2(\DJ(\K))$, that is, the part of $\pi_*(\DJ(\K))$ generated by iterated Whitehead products formed from the canonical elements $t_1,\dots,t_m$. 
\end{prb}

\begin{prb}\label{prb:main2}
Describe the $\Pi$-subalgebra $S(\K)$ generated by 
$t_1,\ldots,t_m$, that is, the part of $\pi_*(\DJ(\K))$ generated 
by iterated Whitehead products formed from the canonical elements  and compositions with the elements of the homotopy groups of spheres.
\end{prb}

These questions are interesting already in the simplest case, where $\K=\L$ is the disjoint union of $m$ points and $\DJ(\L)=(\CC P^\infty)^{\vee m}$ is the $m$-fold wedge.

%The Hilton--Milnor Theorem (see \S \ref{subsec:hilton-milnor}) implies that $S(\K)$ is identified with the image of the homomorphism $\pi_*((S^2)^{\vee m})\to\pi_*(\DJ(\K))$ induced by the inclusion of the wedge $(S^2)^{\vee m}$ into $\DJ(\K)$.

It is helpful to consider the homotopy fibration
\begin{equation}
\label{eq:zk-dj fibration}         
  \ZK\overset{i}\longrightarrow\DJ(\K)\overset{p}
  \longrightarrow(\CC P^\infty)^m,
\end{equation}
where $p$ is the standard inclusion and $\ZK=(D^2,S^1)^\K$ is the \emph{moment-angle complex} corresponding to $\K$. Since the looped map $\Omega p$ has a right homotopy inverse, we obtain an isomorphism of graded abelian groups  
\begin{equation}
\label{eq:pi dj = pi bt + pi zk}    
  \pi_*(\DJ(\K))\cong\pi_*(\CC P^\infty)^{\oplus m}\oplus\pi_*(\ZK).
\end{equation}
However, this isomorphism does not preserve the Whitehead product, and therefore does not give a splitting of $\pi_*(\DJ(\K))$ as a quasi-Lie algebra. In view of Problems~\ref{prb:main1} and~\ref{prb:main2}, it is therefore important to describe the Whitehead products of the canonical elements $t_1,\ldots,t_m\in\pi_2(\DJ(\K))$ with elements from~$\pi_*(\ZK)$.

Although in general the homotopy type of $\ZK$ might be complicated, in many cases $\ZK$ is homotopy equivalent to a wedge of simply connected spheres, and the fibre inclusion $i\colon\ZK\to\DJ(\K)$ can be described as a wedge of iterated (higher) Whitehead products of the canonical elements $t_1,\dots,t_m\in\pi_2(\DJ(\K))$. This holds, for example, for the following classes of simplicial complexes:
\begin{enumerate}
    \item $\K$ is a flag complex with chordal $1$-skeleton \cite[Theorem 4.6]{gptw};
    \item $\K$ is totally homology fillable \cite[Corollary 7.12]{ik_whitehead};
    \item $\K$ is a substitution complex corresponding to an iterated higher Whitehead product \cite[Theorem 5.3]{abramyan}, \cite[Example 7.2]{abramyan_panov}.
\end{enumerate}

If $\ZK\simeq\bigvee_{x\in X}S^{n_x}$ is a wedge of spheres, the Hilton--Milnor theorem~\cite{hilton} provides an additive description of the homotopy groups of~$\ZK$. In this case, the Whitehead product of any two elements of $\pi_*(\DJ(\K))$ can be calculated once we know the elements $[t_i,x]\in\pi_{n_x+1}(\ZK)$, as described in \S\ref{subsec:mac-wedge-of-spheres}.

We also note that in some cases $\Omega\ZK$ is homotopy equivalent to a finite-type product of loops on spheres (by \cite{stanton, st-pseudomanifold, st-poincare, eldridge}, see \cite[Lemma 6.1]{vylegzhanin}). In particular, $\pi_*(\ZK)$ can be described in terms of the homotopy groups of spheres for any flag complex $\K$ by \cite[Theorem 1.2]{vylegzhanin}. After localisation away from a finite number of primes, this holds for any complex $\K$, see \cite{stanton-vylegzhanin}.

\medskip

The case of flag simplicial complexes is of particular importance ($\K$ is \emph{flag} if any set of pairwise connected vertices spans a simplex). This is precisely the case when no higher Whitehead products appear in $\pi_*(\DJ(\K))$. In the flag case, $\ZK$ is homotopy equivalent to a wedge of spheres precisely when the 1-skeleton of $\K$ is a \emph{chordal} graph \cite[Theorem 4.6]{gptw}.
In this case, we have 
\[
  \ZK\simeq\bigvee_{x\in GPTW}S^{|x|}
  %=\bigvee_{J\subset[m]}(S^{|J|+1})^{\vee \dim\H_0(\K_J)}.
\] 
and the inclusion $\ZK\to\DJ(\K)$ is identified with the wedge sum of maps 
\[
  \{S^{|x|}\to\DJ(\K),\, x\in GPTW\}
\]
representing iterated Whitehead products of $t_1,\dots,t_m$ indexed by the \emph{Grbi\'c--Panov--Theriault--Wu elements}, which are certain iterated commutators without repeating indices (see~\S\ref{GPTWpar} for details). 
 
In the chordal flag case, by the Hilton--Milnor Theorem applied to the wedge decomposition of $\ZK$ and \eqref{eq:pi dj = pi bt + pi zk}, $\pi_*(\DJ(\K))$ is a direct sum of homotopy groups of spheres, namely,
\[
  \pi_*(\DJ(\K))\cong
  \ZZ^m\oplus\pi_*\Bigl(\bigvee_{x\in GPTW}S^{|x|}\Bigr)
  \cong\ZZ^m\oplus\bigoplus_{b\in\BC(GPTW)}\pi_*(S^{|b|}).
\]
Here $\BC(GPTW)$ are the \emph{basic commutators} forming the standard basis of the free Lie algebra on the set $GPTW$. In more detail,
\begin{equation}
\label{eq:homotopy groups of dj for wedges}
  \pi_*(\DJ(\K))\cong\ZZ\langle t_1,\dots,t_m\rangle
  \oplus\bigoplus_{b\in\BC(GPTW)} w_b\circ \pi_*(S^{|b|})
\end{equation}
where, if $b\in\BC(GPTW)$, then $w_b\in \pi_{|b|}(\DJ(\K))$ is the corresponding iterated Whitehead product of the composite maps $S^{|x|}\hookrightarrow\ZK\to\DJ(\K)$, $x\in GPTW$.

In particular, for flag complexes $\K$ with chordal $1$-skeleton the GPTW elements in $\pi_*(\DJ(\K))$ do not satisfy any algebraic relations, except for the universal relations that hold in the homotopy groups of any topological space. For general $\K$, the GPTW elements are subject to Lie relations corresponding to chordless cycles, i.\,e., to elements of the first homology groups $H_1(\K^f_J;\k)$ of the flagifications of full subcomplexes $\K_J$. These relations are described in Section~\ref{sec:identities}.

\medskip

To approach Problems~\ref{prb:main1} and~\ref{prb:main2}
for general $\K$, we develop the commutator calculus of iterated Whitehead products. In Section~\ref{sec:computations of brackets} 
we express any iterated Whitehead product~$x$ of the canonical elements $t_1,\ldots,t_m\in\pi_2(\DJ(\K))$ through GPTW elements. If no $t_i$ appears in $x$ twice (there are \emph{no repeating indices}), then $x$ can be written as a Lie polynomial in GPTW elements according to~~\cite[Theorem~4.3]{gptw}. When there are repeating indices, one needs to compose GPTW elements with the Hopf elements $\eta^k\in\pi_{n+k}(S^n)$. In the most basic case, this is expressed by the key identity 
\[
  [t_i,[t_i,t_j]]=[t_i,t_j]\circ\eta\in\pi_4(\DJ(\K))
\] 
(Proposition \ref{prp:iij}). The general case is summarised in the following result.

\begin{thm}[Theorem~\ref{thm:iterated whitehead product to GPTW}]
Let $x\in\pi_{|x|}(\DJ(\K))$ be an iterated Whitehead product of $t_1,\dots,t_m$, $|x|>2$. 
\begin{enumerate}
  \item[(a)] If no $t_i$ appears in $x$ twice, then $x$ is a Lie polynomial in GPTW elements.
  \item[(b)] In general, $x$ is a linear combination of elements of the form $y\circ\eta^k$, where $y$ is an iterated Whitehead product of GPTW elements, and $0\leq k\leq 3$.
\end{enumerate}
\end{thm}

This is proved in Section~\ref{sec:computations of brackets} through a sequence of lemmas, allowing for an algorithmic expression of $x$ in terms of GPTW and Hopf elements.

\medskip

To address Problem~\ref{prb:main2}, in Section \ref{sec:homotopy groups of dj} we identify $S(\mathcal K)$ with the image of the homomorphism induced by the inclusion of the canonical elements and the map $  g_\K\colon\bigvee_{x\in GPTW}S^{|x|}\to \DJ(\K)$
given by the iterated Whitehead products corresponding to GPTW elements. The part of $\pi_*(\DJ(\K))$ generated by iterated Whitehead products is identified in the following result.

\begin{thm}[Proposition~\ref{SKequiv} and Theorem \ref{thm:Phi inj surj criterion}]
\label{thm:Phi inj surj criterion intro}
Consider the homomorphism
\[
  \varPhi\colon \ZZ\langle t_1,\dots,t_m\rangle\oplus \pi_*\Bigl(\bigvee_{x\in GPTW}S^{|x|}\Bigr) \to \pi_*(\DJ(\K))
\] 
given by iterated Whitehead products corresponding to GPTW elements. Then
\begin{enumerate}
 \item $S(\K)=\Img\varPhi$.
 \item $\varPhi$ is surjective if and only if $\K$ is flag;
 \item $\varPhi$ is injective if and only if $\K^1$ is chordal.
\end{enumerate}
\end{thm}

Restating this in terms of $\ZK$ through decomposition~\eqref{eq:pi dj = pi bt + pi zk}, we obtain that $\pi_*(\ZK)$ is generated as a $\Pi$-algebra by GPTW elements if and only if $\K$ is flag  and there are no relations
between the GPTW elements if and only if $\K^1$ is chordal.
Elaborating on Theorem~\ref{thm:Phi inj surj criterion}, we show that $\K$ is flag if and only if $\pi_*(\DJ(\K))$ is generated by $t_1,\dots,t_m$ as a $\Pi$-algebra.

\begin{thm}[Theorem \ref{thm:flagness is surjectivity}]
\label{thm:flag case, full generation}
Let $\K$ be a simplicial complex on $[m]$. The following conditions are equivalent:
\begin{itemize}
\item[(a)] $\K$ is a flag complex;
\item[(b)] the inclusion $(S^2)^{\vee m}\to\DJ(\K)$ induces a surjection on homotopy groups;
\item[(c)] the group $\pi_*(\DJ(\K))$ is generated by elements of the form $x\circ\alpha$, where $x\in\pi_{|x|}(\DJ(\K))$ is an iterated Whitehead product of $t_1,\dots,t_m$, and $\alpha\in\pi_*(S^{|x|})$.
\end{itemize}
\end{thm} 

For $\K=\L$, Problem \ref{prb:main2} can be thought of as determining the image of the standard inclusion $(S^2)^{\vee m}\to (\CC P^\infty)^{\vee m}$ on the level of homotopy groups. More generally, we consider the map of polyhedral products $t_\K\colon(S^2,\ast)^\K\to(\CC P^\infty,\ast)^\K=\DJ(\K)$ induced by the inclusions $S^2\to\CC P^\infty$. Then the homotopy fibration~\eqref{eq:zk-dj fibration} can be included in a diagram of homotopy fibrations:
\begin{equation}
\label{eq:main diagram for dj}
\xymatrix{
(C\Omega S^2,\Omega S^2)^\K
\ar[d]^-{r_{\K}}
\ar[r]^-{}
&
(S^2,\ast)^\K
\ar[r]
\ar[d]^-{t_\K}
&
(S^2)^{\times m}
\ar[d]\\
(CS^1,S^1)^\K
\ar@<1ex>@/^/[u]^-{q_{\K}}
\ar[r]^-{}
&
(\CC P^\infty,\ast)^\K
\ar[r]
&
(\CC P^\infty)^{\times m},
}
\end{equation} 
where the map $r_{\K}$ has a right homotopy inverse $q_{\K}$ with both induced by the natural retraction $S^1\to\Omega S^2\to S^1$. 

When  $\K=\L$, both $\ZK$ and $(C\Omega S^2,\Omega S^2)^\K$ are finite type wedges of spheres, and \eqref{eq:main diagram for dj} takes the following form:
\[
\xymatrix{
\bigvee\limits_{\alpha\in\Zm}(S^{|\alpha|+1})^{\vee (|\supp\alpha|-1)}
\ar[d]^-{r_{\L}}
\ar[rr]^-{f_\L}
&&
(S^2,\ast)^{\vee m}
\ar[r]
\ar[d]
&
(S^2)^{\times m}
\ar[d]\\
\bigvee\limits_{J\subset[m]}(S^{|J|+1})^{\vee (|J|-1)}
\ar@<1ex>@/^/[u]^-{q_{\L}}
\ar[rr]^-{g_\L}
&&
(\CC P^\infty,\ast)^{\vee m}
\ar[r]
&
(\CC P^\infty)^{\times m}.
}
\]
Here, $g_\L$ is a wedge of iterated Whitehead products corresponding to GPTW elements for $\mathcal Z_\L=(CS^{1},S^{1})^{\L}$ (Proposition~\ref{prp:chordal+flag=homotopy type of ZK}).
The other maps in this diagram are described in Section~\ref{sec:homotopy groups of dj}. 

\begin{prp}[see Proposition \ref{diagDJL}]
In the diagram above:
\begin{enumerate}
\item $q_\L$ is the inclusion of wedge summands corresponding to squarefree multi-indices;

\item if $\alpha=J$ is a squarefree multi-index, then the restriction of the map $r_\L$ to $(S^{|J|+1})^{\vee (|J|-1)}$ is the identity map;

\item each restriction $S^{|\alpha|+1}\to\bigvee S^{|J|+1}$ of $r_\L$ is an explicit linear combination of elements of the form $s\circ\eta^k$ and $[s',s'']\circ\eta^k$, where $s,s',s''$ are inclusions of wedge summands $S^{|J|+1}$ and $\eta^k\colon  S^{|w|+k}\to S^{|w|}$ is the iterated Hopf map;%, $k\geq 0$;

\item $r_\L$ is nontrivial only on a finite number of wedge summands.
\end{enumerate}
\end{prp}

Problem~\ref{prb:main1} is further addressed in Section~\ref{sec:ql(K)}. For $n\geq 2$, let $A_n\subset\pi_*(S^n)$ be the quasi-Lie subalgebra generated by the identity map $\iota\colon S^n\to S^n$, compositions with suspended Hopf elements, and Whitehead products. The additive structure of $A_n$ is known, see Appendix~\ref{sec:groups A_n}. In particular, $A_n$ is a finite direct sum of copies of $\ZZ$, $\ZZ_2$ and $\ZZ_3$.
Let $Q(\K)$ be the subgroup
\[
  Q(\K)=\ZZ\langle t_1,\ldots,t_m\rangle
  \oplus\bigoplus_{b\in\BC(GPTW)}A_{|b|}
\]  
of the group
\[
  \ZZ\langle t_1,\ldots,t_m\rangle
  \oplus\bigoplus_{b\in\BC(GPTW)}\pi_*(S^{|b|})\cong \ZZ^m\oplus\pi_*\Big(\bigvee_{x\in GPTW}S^{|x|}\Big),
\]
and let 
\[
  \phi\colon Q(\K)\to\pi_*(\DJ(\K))
\]
be the restriction of the map $\varPhi\colon \ZZ^m\oplus\pi_*(\bigvee_{x\in GPTW} S^{|x|})\to\pi_*(\DJ(\K))$.

The quasi-Lie algebra $QL(\K)$ is described as follows.

\begin{thm}[Theorem \ref{thm:QL(K) description}]
\label{thm:QL(K) description intro}
Let $\K$ be a simplicial complex. Then:
\begin{enumerate}
  \item $QL(\K)=\Img\phi$, i.\,e. $QL(\K)\subset\pi_*(\DJ(\K))$ is additively generated by the elements $t_1,\dots,t_m$ and $w_b\circ\alpha$, where $b\in\BC(GPTW)$ and $\alpha\in A_{|b|}$;
  
  \item the natural homomorphism of quasi-Lie algebras $QL(\K)\to QL(\K^f)$ induced by the flagification $\K\to\K^f$ is surjective;
  
  \item if $\K^1$ is a chordal graph, then $Q(\K)\cong QL(\K)\cong QL(\K^f)$. In particular, in this case $QL(\K)$ has no $p$-torsion for $p>3$.
\end{enumerate}
\end{thm}

We also describe an algorithm that allows us to express any element of $QL(\K)$ as an element in $\Img\phi$ and to compute the bracket in $QL(\K)$ in these terms.

\medskip

Problems~\ref{prb:main1} and~\ref{prb:main2} can be generalised. The homotopy fibration \eqref{eq:zk-dj fibration} is a special case of the homotopy fibration of polyhedral products
\[
  (C\Omega\X,\Omega\X)^\K\to(\X,\ast)^\K\to\prod X_i.
\]  
The canonical generators $t_{i}$ generalise to the composites
\[
  t_{\X,i}\colon \Sigma\Omega X_i\xrightarrow{\ev_{X_i}} X_i\hookrightarrow (\X,\ast)^\K 
\]  
where $\ev_X\colon \Sigma\Omega X\to X$ is the standard evaluation map and the right map is induced by including the vertex $i$ into $\K$. In Section \ref{sec:spheres and dj} we study the iterated generalised Whitehead products of the maps $t_{\X,i}$. 

Let $E_Y\colon Y\to\Omega\Sigma Y$ be the suspension map, defined as the adjoint of the identity map on $\Sigma Y$.
Diagram~\eqref{eq:main diagram for dj} generalises to the commutative diagram of fibrations
\[
\xymatrix{
(C\Omega\Sigma\Omega\X,\Omega\Sigma\Omega\X)^\K
\ar[d]^-{r_{\K}=(C\Omega\ev_\X,\Omega\ev_\X)^\K}
\ar[rr]
&&
(\Sigma\Omega\X,\ast)^\K
\ar[d]^-{(\ev_\X,\ast)^\K}
\ar[r]
&
\prod_{i=1}^m\Sigma\Omega X_i
\ar[d]^-{\prod \ev_\X}
\\
(C\Omega\X,\Omega\X)^\K
\ar@<1ex>@/^/[u]^-{q_{\K}=(CE,E)^\K}
\ar[rr]
&&
(\X,\ast)^\K
\ar[r]
&
\prod_{i=1}^m X_i.
}
\]

When $\K=\L$, the fibre inclusions in the diagram above are identified with wedges of generalised Whitehead products:
\[
  \xymatrix
  {
  \bigvee_{\alpha\in\Zm}
  %\bigvee_{j\in \Theta_\L(\alpha)}
  \bigl(\Sigma(\Omega\X)^{\wedge \alpha}\bigr)^{\vee (|\supp\alpha|-1)}
  \ar[r]%^-{\vee c(\alpha-j,j;\incl)}
  \ar[d]^{r'_\L}
  &
  \bigvee_{i=1}^m \Sigma\Omega X_i
  \ar[r]
  \ar[d] 
  &
  \prod_{i=1}^m\Sigma\Omega X_i\ar[d]\\
  \bigvee_{J\subset[m]}
  %\bigvee_{j\in\Theta_\L(J)}
  \bigl(\Sigma(\Omega\X)^{\wedge J}\bigr)^{\vee (|J|-1)}
  \ar[r]%^-{\vee c(J\setminus j,j;t)}
  \ar@<1ex>@/^/[u]^-{q'_{\L}} 
  %\ar@/_1pc/[u]_-{q_\L}
  &
  \bigvee_{i=1}^m X_i\ar[r]
  &
  \prod_{i=1}^m X_i,
  }
\]
as described in Theorem~\ref{thm:fiber of X^L} and Proposition~\ref{prp:f and s for X^L}.

Generalising the identity $[[t_i,t_j],t_j]=[t_i,t_j]\circ\eta$, we obtain
\[               
  [[[t_{\X,i},t_{\X,j}],t_{\X,j}]\simeq
  [t_{\X,i},t_{\X,j}]\circ(\id_{  \Omega X_i}\wedge H\mu_{X_j})
  \colon \Sigma(\Omega X_i\wedge\Omega X_j\wedge\Omega X_j)\to (\X,\ast)^\K,
\]
i.e. the role of the Hopf element $\eta\in\pi_3(S^2)$ is played by the Hopf construction
\[
  H\mu_X\colon \Sigma(\Omega X\wedge\Omega X)\to\Sigma\Omega X
\]  
of the multiplication $\mu_X\colon \Omega X\times\Omega X\to\Omega X$. This formula was mentioned by Baues~\cite[(A.1.23)]{baues_homotopy}; we give a complete proof in Appendix~\ref{sec:baues_formula}. However, the general case is much less approachable: Davis--Januszkiewicz spaces are simpler because $\Omega X$ is a co-H-space when $X=\CC P^\infty$. On the other hand, our results easily generalise from the $\CC P^\infty$ case to the $\HH P^\infty$ case.

\section{Preliminaries} 
This section collects together many of the tools that will be used subsequently and describes their relevant properties.

\subsection{Polyhedral products}
Let $\K$ be a simplicial complex on the set
$[m]=\{1,2,\ldots,m\}$. Let $(\underline X,\underline A)=\{(X_{i},A_{i})\}_{i=1}^{m}$ be the sequence of pairs of pointed $CW$-complexes, $A_{i}\subset X_{i}$.
For each simplex $I=(i_1,\ldots,i_k)\in\K$, let $(\underline X,\underline A)^I$ be the
subspace of $\prod_{i=1}^{m} X_{i}$ defined by
\[
(\underline X,\underline A)^I=\prod_{i=1}^{m} Y_{i}\qquad
       \mbox{where}\qquad Y_{i}=\left\{\begin{array}{ll}
                                             X_{i} & \mbox{if $i\in I$} \\
                                             A_{i} & \mbox{if $i\notin I$}.
                                       \end{array}\right.
\]
The \emph{polyhedral product} determined by $(\underline X,\underline A)$ and $\K$ is
\[
  (\underline X,\underline A)^\K=\bigcup_{I\in\K}
  (\underline X,\underline A)^I \subset\prod_{i=1}^{m} X_{i}.
\]
For example, suppose each $A_{i}$ is a point. If $\K$ is a disjoint
union of $m$ points then $(\underline{X},\ast)^{\K}$ is
the wedge $X_{1}\vee\cdots\vee X_{m}$, and if $\K$ is the standard
$(m-1)$-simplex then $(\underline{X},\ast)^{\K}$ is the
product $X_{1}\times\cdots\times X_{m}$.

The \emph{moment-angle complex} is the polyhedral product $\ZK=(D^2,S^1)^\K$ and the \emph{Davis--Januszkiewicz space} is $\DJ(\K)=(\CC P^\infty,\ast)^K$.

\begin{prp}\label{loopsec}
There is a homotopy fibration
\[
  (C\Omega\X,\Omega\X)^\K \to (\X,\ast)^\K \to \prod_{i=1}^m X_i.
\]
The associated principal homotopy fibration of loop spaces
\[
  \Omega(C\Omega\X,\Omega\X)^\K \to \Omega(\X,\ast)^\K \to \prod_{i=1}^m \Omega X_i
\]
admits a section $\prod_{i=1}^m \Omega X_i\to\Omega(\X,\ast)^\K$ and is therefore trivial.
\end{prp}

\begin{proof} 
The existence of the homotopy fibration is a special case of~\cite[Lemma 2.3.1]{denham_suciu}. The inclusion of the vertex set into $\K$ induces a map of polyhedral products $\bigvee_{i=1}^{m} X_{i}\longrightarrow (\X,\ast)^{\K}$ with the property that the composite 
$\alpha\colon\bigvee_{i=1}^{m} X_{i}\longrightarrow (\X,\ast)^{\K}\longrightarrow\prod_{i=1}^{m} X_{i}$ 
is the inclusion of the wedge into the product. Porter~\cite{porter} showed that $\Omega\alpha$ has a right homotopy inverse. Therefore the map 
$\Omega(\X,\ast)^{\K}\longrightarrow\prod_{i=1}^{m} \Omega X_{i}$ has a right homotopy inverse.
\end{proof}

As an example, taking each $X_{i}=\mathbb{C}P^{\infty}$ and noting that there are homotopy equivalences of pairs $(C\Omega\mathbb{C}P^{\infty},\mathbb{C}P^{\infty})\simeq (CS^{1},S^{1})\simeq (D^{2},S^{1})$, we obtain the homotopy fibration~\eqref{eq:zk-dj fibration}.

\subsection{Whitehead products}
Given two maps $f\colon S^{k+1}\to Z$ and $g\colon S^{l+1}\to Z$, their \emph{Whitehead product} is the composite
\[
  [f,g]\colon S^{k+l+1}\longrightarrow S^{k+1}\vee S^{l+1}\xrightarrow{f\vee g} Z, 
\]
where the first map is the attaching map of the top cell in the product $S^{k+1}\times S^{l+1}$.

The adjoint $[f,g]'\colon S^{k+l}\to\Omega Z$ is homotopic up to sign to the composite
\[
  S^{k+l}\cong S^k\wedge S^l\xrightarrow{f'\wedge g'}
  \Omega Z\wedge\Omega Z\xrightarrow{c}\Omega Z,
\]
where $f'\colon S^k\to\Omega Z$ and $g'\colon S^l\to\Omega Z$ are the adjoints of $f$ and $g$ respectively, and $c\colon \Omega Z\wedge\Omega Z\to\Omega Z$ is the commutator, $(x,y)\mapsto x^{-1}y^{-1}xy$. The sign depends on the identification of $S^{k+l}$ with $S^k\wedge S^l$; we use the sign convention from~\cite{whitehead}.

The Whitehead product of maps defines an operation on homotopy groups,
\[
  \alpha\in\pi_{k+1}(X),~\beta\in\pi_{\ell+1}(X)~\leadsto~[\alpha,\beta]\in\pi_{k+\ell+1}(X),
\]
which is also referred to as the Whitehead product. We denote by $|\alpha|$ the topological degree: $|\alpha|=i$ if $\alpha\in\pi_i(X)$.

\begin{prp}[{\cite[X.7.5,12,13,14]{whitehead}}]
\label{prp:whitehead_properties}
The Whitehead product has the following properties:
\begin{enumerate}
\item $[\alpha,\beta]=(-1)^{|\alpha|\cdot|\beta|}[\beta,\alpha];$

\item $[\alpha+\alpha',\beta]=[\alpha,\beta]+[\alpha',\beta];$

\item $(-1)^{|\alpha|\cdot|\gamma|}[[\alpha,\beta],\gamma]+(-1)^{|\alpha|\cdot|\beta|}[[\beta,\gamma],\alpha]+(-1)^{|\beta|\cdot|\gamma|}[[\gamma,\alpha],\beta]=0.$
\qed
\end{enumerate}
\end{prp}

The (generalised) \emph{Whitehead product} of based maps $f\colon\Sigma X\to Z$ and $g\colon\Sigma Y\to Z$ is the based map $[f,g]\colon \Sigma(X\wedge Y)\to Z$ defined as the adjoint of the composite
\[
  X\wedge Y\xrightarrow{f'\wedge g'}
  \Omega Z\wedge\Omega Z\xrightarrow{c}\Omega Z.
\]
It coincides with the classical Whitehead product when $X=S^k$ and $Y=S^l$, with the appropriate sign convention.

\subsection{GPTW elements}\label{GPTWpar}
Consider the canonical elements $t_i\colon S^2\to\DJ(\K)=(\CC P^\infty,\ast)^\K$ for $i=1,\ldots,m$.

A \emph{GPTW element} is an iterated Whitehead product of the form 
\[
  [t_{i_1},[t_{i_2},\dots[t_{i_k},t_j]\dots]]\colon S^{k+2}\to\DJ(\K),
\]  
where $k>0$, $i_1<\cdots<i_k>j$, $j\notin \{i_1,\ldots,i_k\}$ and $j$ is the smallest vertex in a connected component not containing~$i_k$ of the subcomplex~$\mathcal K_{\{i_1,\ldots,i_k,j\}}$. For a given $\K$, denote the set of all GPTW elements by $GPTW$. Then
\[
  |GPTW|=\sum_{J\subset[m]}\widetilde b_0(\K_J),
\]  
where $\widetilde b_0(\K_J)$ is one less the number of connected components of~$\K_J$ and $\widetilde b_0(\varnothing)=0$.

A GPTW element lifts uniquely to a map $S^{k+2}\to\ZK$ through the homotopy fibration~\eqref{eq:zk-dj fibration}. The GPTW elements are viewed as elements in both $\pi_{k+2}(\DJ(\K))$ and $\pi_{k+2}(\ZK)$ with $k>0$. 

\begin{thm}[{\cite[Theorem~4.3, 4.6]{gptw},
\cite[Theorem~5.1]{vylegzhanin}}] 
\label{gptwvylegzhanin}
Let $\mathcal K$ be a flag simplicial complex on the vertex set~$[m]$ and $\mathbf k$ a commutative ring with unit.

\smallskip

\noindent{\rm(1)} There is an isomorphism of algebras
\[
  H_*(\Omega\DJ(\K);\mathbf k) 
  \cong T(u_1,\dots,u_m)/(u_i^2=0,\;u_iu_j+u_ju_i=0
  \text{ for }\{i,j\}\in\K),
\]
where $T(u_1,\ldots,u_m)$ denotes a free associative algebra on the generators $u_1,\ldots,u_m$, and $u_i\in H_1(\Omega\DJ(\K);\k)$ is the adjoint of $t_i\colon S^2\to\DJ(\K)$. 

The rational homotopy Lie algebra of $\DJ(\K)$ is given by
\[
  \pi_*(\Omega\DJ(\K))\otimes_{\ZZ}\QQ\cong
  \FL_\QQ\langle u_1,\ldots,u_m\rangle
  ([u_i,u_i]=0,\;[u_i,u_j]=0 \text{ for }\{i,j\}\in\K).
\]

\noindent{\rm (2)} The loop homology algebra $H_*(\Omega\ZK;\mathbf k)$ is minimally generated by the adjoints of the GPTW elements.

\smallskip

\noindent{\rm (3)} If the $1$-skeleton of $\K$ is a chordal graph, then there is a homotopy equivalence
\[
  \ZK\simeq \bigvee_{x\in GPTW}S^{|x|}
  =\bigvee_{J\subset[m]}
  (S^{|J|+1})^{\vee\widetilde b_0(\K_J)}. 
\] 
\end{thm} 

A proof of Theorem~\ref{gptwvylegzhanin}~(3) is included as a consequence of the stronger Proposition~\ref{prp:chordal+flag=homotopy type of ZK}, which also describes the map $\ZK\longrightarrow\DJ(\K)$.

\subsection{Basic commutators}
Let $X=\{x_1>\dots>x_m\}$ be an ordered set. Choose any linear ordering on the set of all iterated commutators on $X$ so that the following condition holds:
\begin{itemize}
\item If $c_1$ is longer than $c_2,$ then $c_1>c_2.$
\end{itemize}
Now define the set $\BC(X)$ of \emph{basic commutators}~\cite{hilton, hilton_quasilie} on $X$ as follows:
\begin{itemize}
\item $x_1,\dots,x_m\in\BC(X);$
\item $c=[c',c'']\in\BC(X)$ if and only if the following conditions hold:
\begin{enumerate}
\item $c',c''\in\BC(X);$
\item $c'<c'';$
\item if $c''=[c_1,c_2],$ then $c'\geq c_1.$
\end{enumerate}
\end{itemize}

For example, $[[x_i,x_j],x_k]$ is never a basic commutator, since $[x_i,x_j]>x_k$. On the other hand, for nested commutators we have
$$[x_{i_1},[\dots[x_{i_{k-2}},[x_{i_{k-1}},x_{i_k}]]\dots]]\in\BC(X)$$
if and only if $i_1\leq\dots\leq i_{k-1}>i_k$. It follows that $GPTW\subset\BC(t_1,\dots,t_m)$.

It is well known (see~\cite[Theorem 3.2]{hilton}, for example) that basic commutators form a basis of the free Lie ring $\FL(x_1,\dots,x_m)$ generated by $X=\{x_1,\dots,x_m\}$, so that $\FL(X)\cong\ZZ\langle\BC(X)\rangle$ as free abelian groups.

\begin{rmk}
Our convention $x_1>\dots>x_m$ follows Hilton's papers~\cite{hilton,hilton_quasilie}. Note that G.\,W.~Whitehead's book \cite[XI.\S 6]{whitehead} uses the opposite ordering.
\end{rmk}

\subsection{Notation for iterated Whitehead products}\label{notiwp}
Let $\{e_i\}_{i=1}^{m}$ be the standard basis of $\mathbb{Z}^m$. 
We consider multi-indices $\alpha=\sum_{i=1}^m\alpha_ie_i\in\Zm$ and denote 
\[
  |\alpha|:=\sum\alpha_i,\quad \supp\alpha:=\{i\in [m]:\alpha_i\neq 0\}.
\]
A subset $I\subset[m]$ is identified with the multi-index $\sum_{i\in I}e_i$, so $I=\supp I$. We also write $\alpha-i:=\alpha-e_i$ for $i\in\supp\alpha$.

For $\X=(X_1,\dots,X_m)$, we denote $\X^{\wedge\alpha}:=\bigwedge_{i=1}^m X_i^{\wedge\alpha_i}$, so $\X^{\wedge I}=\bigwedge_{i\in I}X_i$ and $\X^{\wedge\alpha}\wedge\X^{\wedge\beta}=\X^{\wedge(\alpha+\beta)}$ for any $\alpha,\beta\in\Zm$, where we formally set $\X^{\wedge 0}:=S^0$.

Let $f_i\colon\Sigma X_i\to Z$, $i=1,\dots,m$, be based maps. For a multi-index $\alpha\in\Zm$ and an index $j\in[m]$, we consider the iterated generalised Whitehead product
\[
  c(\alpha,j;\underline{f}):=[\underbrace{f_1,[\dots,[f_1}_{\alpha_1\text{ times}},[f_2,\dots,[\underbrace{f_m,[\dots[f_m}_{\alpha_m\text{ times}},f_j]\dots]]\dots]]\dots]]
  \colon \Sigma\X^{\wedge\alpha}\wedge X_j\to Z.
\]  
For example, $c(3e_2+e_5,4;\underline{f})=[f_2,[f_2,[f_2,[f_5,f_4]]]].$

For $i\in[m]$ and $\X=(X_1,\dots,X_m)$, we consider the maps
\[
  \incl_{\X,i}\colon X_i\hookrightarrow (\X,\ast)^\K,\quad t_{\X,i}:=\incl_{\X,i}\circ \ev_{X_i}\colon \Sigma\Omega X_i\to X_i\to (\X,\ast)^\K
\]  
and their iterated Whitehead products
\begin{equation}
\label{cprod}
\begin{gathered}
  c(\alpha,j;t_{\X})\colon
  \Sigma(\Omega\X)^{\wedge\alpha}\wedge\Omega X_j\to (\X,\ast)^\K,\\
  c(\alpha,j;\incl_{\Sigma\Y})\colon 
  \Sigma\Y^{\wedge\alpha}\wedge Y_j\to (\Sigma\Y,\ast)^\K.
\end{gathered}
\end{equation}
In the case $X_i=\CC P^\infty$ we omit the subscript $\X$ and obtain the maps
\[
  t_i\colon S^2\cong\Sigma\Omega\CC P^\infty\to\DJ(\K),\quad c(\alpha,j;\underline t)\colon S^{|\alpha|+2}\cong\Sigma S^{|\alpha|}\wedge S^1\to\DJ(\K).
\]  
Similarly, for $Y_i=S^1$ we denote $s_i=\incl_{\Sigma S^1,i}$ and obtain the maps
\begin{equation}\label{cscom}
  s_i\colon S^2\to (S^2,\ast)^\K,\quad c(\alpha,j;\underline{s})\colon S^{|\alpha|+2}\cong\Sigma S^{|\alpha|}\wedge S^1\to (S^2,\ast)^\K.
\end{equation}

For a simplicial complex $\K$ on $[m]$ and a subset $J\subset[m]$, denote by $\Theta_\K(J)$ the $\widetilde b_0(\K_J)$-element subset of $J$ which is used in the definition of GPTW elements: $j\in \Theta_\K(J)$ if and only if $j$ and $\max(J)$ are in different connected components of $\K_J$ and $j$ is the smallest element in its connected component of~$\K_J$. A GPTW element in $\DJ(\K)$ is therefore $c(J\setminus j,j;\underline t)\colon S^{|J|+1}\to\DJ(\K)$ with $J\subset[m]$ and $j\in\Theta_\K(J)$. We define the set of GPTW elements in the polyhedral product $(\X,\ast)^\K$ as
\begin{equation}\label{GPTWTheta}
  GPTW=\{c(J\setminus j,j;t_\X)\colon  J\subset[m], j\in\Theta_\K(J)\}.
\end{equation}
Clearly, $\Theta_\K(\varnothing)=\varnothing$ and $|\Theta_\K(J)|=\widetilde b_0(\K_J)$ for $J\neq\varnothing$. For the discrete $m$ point complex $\L$ we have $\Theta_\L(J)=J\setminus\{\max(J)\}$.

We also denote $\Theta_\K(\alpha):=\Theta_\K(\supp\alpha)$ and $\K_\alpha:=\K_{\supp\alpha}$ for $\alpha\in\Zm$.

\subsection{Quasi-Lie rings} 
For a graded module $A$ and an element $a\in A$, write $\deg(a)$ for the degree of $a$ and $A_{even}, A_{odd}$ for the submodules of $A$ consisting of elements of even or odd degree respectively.

\begin{dfn}\label{dfn:quasi-Lie ring}
A \emph{graded quasi-Lie ring} is a graded abelian group $L=\bigoplus_{n\geq 0} L_n$ with a bilinear bracket $[-,-]:L_i\otimes L_j\to L_{i+j}$ satisfying the identities
\begin{gather*}
  [a,b]=(-1)^{\deg(a)\deg(b)+1}[b,a];\\
  (-1)^{\deg(a)\deg(c)}[a,[b,c]]+(-1)^{\deg(b)\deg(a)}[b,[c,a]]+(-1)^{\deg(c)\deg(b)}[c,[a,b]]=0,
\end{gather*}
where $\deg(a)=n$ if $a\in L_n$.

Graded quasi-Lie algebras over a commutative ring $\k$ are defined similarly.
\end{dfn}

\begin{rmk}
The first identity in Definition~\ref{dfn:quasi-Lie ring} implies that $2[a,a]=0$ if $a\in L_{even}.$ It follows that if $2\in\k$ is invertible, then any graded quasi-Lie algebra is a graded Lie (super)algebra. This is not true in general.
\end{rmk}

The next lemma follows from the identities in Definition~\ref{dfn:quasi-Lie ring}.

\begin{lmm}
\label{lmm:quasilie_torsion}
Let $Q$ be a graded quasi-Lie ring.
\begin{enumerate}
\item If $a\in Q_{even},$ then $2[a,a]=0$ and $[a,[a,a]]=0.$
\item If $a\in Q_{odd},$ then $3[a,[a,a]]=0.$\qed
\end{enumerate}
\end{lmm}

The following result of Hilton describes the additive structure of free quasi-Lie rings.

\begin{thm}[{\cite[Theorem 3]{hilton_quasilie}}]
\label{thm:free_quasilie_basis}
Let $\FQL(X)=\FQL(x_1,\dots,x_m)$ be the free graded quasi-Lie ring generated by a set of homogeneous elements $X=\{x_1,\dots,x_m\}.$ Then there is an isomorphism of abelian groups
\[
  \FQL(X)\cong
  \bigoplus_{b\in\BC(X)_{odd}}\bigl(\ZZ\langle b\rangle\oplus 
  \ZZ\langle [b,b]\rangle \oplus 
  \ZZ_3\langle[b,[b,b]]\rangle\bigr)
  \oplus
  \bigoplus_{b\in\BC(X)_{even}}\bigl(\ZZ\langle b\rangle
  \oplus \ZZ_2\langle[b,b]\rangle\bigr).
\]
\end{thm}

\begin{crl}
\label{crl:lie_polynomial_normal_form}
Any element of a graded quasi-Lie ring $L$ generated by homogeneous elements $x_1,\dots,x_m$ is a linear combination of elements of the form $b$, $[b,b]$ and $[b,[b,b]]$, where $b\in\BC(x_1,\dots,x_m).$
\end{crl}

\subsection{The Hilton--Milnor Theorem}
\label{subsec:hilton-milnor}
For based maps $f\colon\Sigma X\to Z$, $g\colon\Sigma Y\to Z$, let $[f,g]\colon\Sigma(X\wedge Y)\to Z$ be their generalised Whitehead product.

Given maps $f_i\colon\Sigma X_i\to Y$, $i=1,\dots,m$, we can associate a map with each iterated commutator of elements $X=\{x_1,\dots,x_m\}$. For example, the commutator $b=[x_2,[x_5,x_1]]$ corresponds to the map 
\[
  b(\underline{f}):=[f_2,[f_5,f_1]]\colon\Sigma X_2\wedge X_5\wedge X_1\to Y.
\]  
In particular, for any $b\in\BC(X)$ we obtain a map $b(\underline{f})\colon\Sigma\underline{X}^{\wedge \mathrm{mdeg}(b)}\to Y,$ where $\underline{X}^{\wedge\alpha}:=\bigwedge_{i=1}^mX_i^{\wedge \alpha_i}$ for $\alpha\in\Zm$, and $\mathrm{mdeg}(b)\in\Zm$ counts the number of occurrences of $x_i$ in $b$.

\begin{thm}[Hilton--Milnor]
\label{thm:hilton-milnor}
Let $X_1,\dots,X_m$ be connected pointed spaces, and let $\incl_k\colon\Sigma X_k\to \bigvee_{i=1}^m \Sigma X_i$ be the canonical inclusions. Then the weak product
\[
  \prod_{b\in\BC(X)}\Omega\Sigma\underline{X}^{\wedge |b|}\to \Omega(\Sigma X_1\vee\dots\vee\Sigma X_m)
\]
of maps $\Omega b(\underline{\incl})\colon\Omega\Sigma\underline{X}^{\wedge |b|}\to \Omega(\Sigma X_1\vee\cdots\vee\Sigma X_m)$
is a weak homotopy equivalence.
\end{thm}

Less precisely, the factors correspond to a basis in a free Lie algebra with generators corresponding to the wedge summands.

In the case $X_i=S^{n_i}$, we obtain a homotopy equivalence
\[
  \Omega(S^{n_1+1}\vee\dots\vee S^{n_m+1})\simeq\prod_{b\in\BC(X)}\Omega S^{1+\sum_i \mdeg(b)_i\cdot n_i},
\]  
implying there is an isomorphism of homotopy groups 
\begin{equation}\label{piwesp}
  \pi_*(S^{n_1+1}\vee\dots\vee S^{n_m+1})\cong\bigoplus_{b\in\BC(X)}\pi_*\bigl(S^{1+\sum_i \mdeg(b)_i\cdot n_i}\bigr)
\end{equation}
where the right side is generated by elements of the form $w_b\circ\alpha$, where $w_b\colon S^{|w_b|}\to S^{n_1+1}\vee\dots\vee S^{n_m+1}$ is the Whitehead product corresponding to a basic commutator $b\in\BC(X)$ and $\alpha\in\pi_*(S^{|w_b|})$.

\subsection{Composition products in homotopy groups}
We always assume that the following formulas only involve $\pi_n$ for $n\geq 2$. Our main references are the books by G.\,W.~Whitehead~\cite{whitehead} and Baues~\cite{baues}.

We consider three homotopy operations:
\begin{enumerate}
\item the suspension homomorphism $$E^k:\pi_*(X)\to\pi_{*+k}(\Sigma^k X),~k\geq 0;$$
\item the Whitehead product
$$\alpha\in\pi_{k+1}(X),~\beta\in\pi_{\ell+1}(X)~\leadsto~[\alpha,\beta]\in\pi_{k+\ell+1}(X);$$
\item the composition product
$$\alpha\in\pi_k(X),~\beta\in\pi_\ell(S^k)
~\leadsto~\alpha\circ\beta\in\pi_\ell(X).$$
\end{enumerate}

Comparing the relations in Proposition~\ref{prp:whitehead_properties} and Definition~\ref{dfn:quasi-Lie ring}, we obtain

\begin{prp}
\label{crl:homotopy_quasilie_ring}
For any simply connected space $X$, the abelian group $\pi_*(X)=\bigoplus_{n\geq 2}\pi_n(X)$ is a quasi-Lie ring with respect to the bracket $\alpha\otimes\beta\mapsto (-1)^{|\alpha|}[\alpha,\beta]$ and the grading $\deg(\alpha):=|\alpha|-1.$\qed
\end{prp}

\begin{rmk}
In our sign convention the Hurewicz homomorphism $h\colon\pi_*(X)\to H_{*-1}(\Omega X;\ZZ)$ preserves the bracket only up to sign. Sometimes in the literature a different sign convention is used, for which $h$ is a homomorphism of quasi-Lie rings.
\end{rmk}

We use the same Greek letter to denote iterated suspensions of the same element in the homotopy groups of spheres; the subscript indicates the codomain. For example, $\eta_2\in\pi_3(S^2)$ is the Hopf element; we have $\pi_3(S^2)=\ZZ\langle\eta_2\rangle$ and $\pi_{k+1}(S^k)=\ZZ_2\langle\eta_k\rangle$ for $k\geq 3$, where $\eta_k:=E^{k-2}\eta_2.$ We omit the subscript when the codomain is clear from the context. The identity map $\id_{S^n}\colon S^n\to S^n$ corresponds to the element $\iota_n\in\pi_n(S^n)$.

For $\alpha\in\pi_{k+i}(S^k)$ we write $\alpha^n:=\underbrace{\alpha\circ\dots\circ\alpha}_{n\text{ times}}\in\pi_{k+in}(S^k).$ For example, $\eta_2^3=\eta_2\circ\eta_3\circ\eta_4\in\pi_5(S^2).$

\begin{prp}[{\cite[X.8.1,2,18]{whitehead}}]
\label{prp:composition_properties}
The composition product has the following properties:
\begin{enumerate}
\item $\alpha\circ\iota_{|\alpha|}=\alpha;$
\item $E(\alpha\circ\beta)=E\alpha\circ E\beta;$
\item $\alpha\circ (\beta+\beta')=\alpha\circ\beta+\alpha\circ\beta';$
\item $(\alpha+\alpha')\circ\beta = \alpha\circ\beta + \alpha'\circ\beta$ if $\beta$ is a suspension;
\item $\alpha\circ[\beta,\beta']=[\alpha\circ\beta,\alpha\circ\beta'];$
\item $[\alpha,\beta\circ\varphi]=[\alpha,\beta]\circ E^{|\alpha|-1}\varphi$ if $\varphi$ is a suspension.\qed
\end{enumerate}
\end{prp}

\subsection{James--Hopf invariants}
\begin{dfn}[{\cite[\S II.2]{baues}}]
Let $X,A$ be pointed spaces, and let $r\geq 1.$ Recall that $J(A)\simeq\Omega\Sigma A$ is a free topological monoid on $A$, hence points of $J(A)$ correspond to ordered tuples $a_1a_2\cdots a_n,$ for $a_j\in A\setminus\{\ast\}$ and $n\geq 0$. Define maps
\[
  g_r\colon J(A)\to J(A^{\wedge r}),~a_1\cdots a_n\mapsto
  \prod_{\{i_1<\dots<i_r\}\subset\{1,\dots,n\}}
  (a_{i_1}\wedge\dots\wedge a_{i_r}),
\]
where the product on the right is taken in lexicographically ascending order. It follows that $g_r(a_1\cdots a_n)=\ast$ if $r>n.$

For example, $g_1=\id\colon J(A)\to J(A),$ and
\[
  g_2(a_1\cdots a_n)=(a_1\wedge a_2)\cdot (a_1\wedge a_3)\cdots (a_{n-1}\wedge a_n)\in J(A\wedge A).
\]  

The \emph{generalised James--Hopf invariants} $\gamma_r\colon[\Sigma X,\Sigma A]\to[\Sigma X,\Sigma A^{\wedge r}]$ are defined as the compositions
\[
  [\Sigma X,\Sigma A]\cong[X,\Omega\Sigma A]\cong [X,J(A)]\overset{(g_r)_*}\longrightarrow[X,J(A^{\wedge r})]\cong [\Sigma X,\Sigma A^{\wedge r}].
\]

In particular, for $X=S^k$ and $A=S^\ell$ we obtain ordinary James--Hopf invariants
\[
  \gamma_r\colon\pi_{k+1}(S^{\ell+1})\to\pi_{k+1}(S^{r\ell+1}),~r\geq 1.
\]
\end{dfn}

\begin{prp}
\label{prp:james_properties}
The James--Hopf invariants $\gamma_r\colon \pi_{k+1}(S^{\ell+1})\to\pi_{k+1}(S^{r\ell+1})$ have the following properties:
\begin{enumerate}
\item $\gamma_1(\alpha)=\alpha;$

\item $\gamma_{r}(\alpha)=0$ if $\alpha$ is a suspension and $r\geq 2$;

\item $\gamma_r(\alpha)=0$ for $r>\lfloor k/\ell\rfloor;$

\item if $k=2\ell,$ then $\gamma_2\colon\pi_{2\ell+1}(S^{\ell+1})\to\pi_{2\ell+1}(S^{2\ell+1})\cong\ZZ$ is (up to sign) the classical Hopf invariant $H_0\colon \pi_{2\ell+1}(S^{\ell+1})\to\ZZ$.
\end{enumerate}
\end{prp}

\begin{proof}
Property (1) follows by definition.

For (2), note that if $\alpha:\Sigma X\to \Sigma A$ is a suspension map, then the corresponding map $X\to J(A)$ factors through the inclusion $A\hookrightarrow J(A)$. Now if $r\ge2$ then $g_r$ is trivial when restricted to the image of $A\hookrightarrow J(A)$ (the $1$-tuples in $J(A)$).

Property (3) is true for dimension reasons.

For (4), see the discussion in Boardman and Steer~\cite[\S7]{boardman_steer}. In their notation, $E^{n-1}\gamma_n(\alpha)=\lambda_n(\alpha)$ \cite[Theorem 3.15]{boardman_steer}, hence $E\gamma_2(\alpha)=\lambda_2(\alpha),$ which (up to sign) is equal to $E\circ H\colon \pi_{k+1}(S^{\ell+1})\to\pi_{k+2}(S^{2\ell+2})$, the suspension of the generalised Hopf invariant $H\colon\pi_{k+1}(S^{\ell+1})\to\pi_{k+1}(S^{2\ell+1})$. This recovers the classical Hopf invariant for $k=2\ell$.
\end{proof}

\subsection{Distributivity laws}
The following important result is originally due to Barcus and Barratt \cite[Corollary~7.4]{barcus_barratt}, who stated it in terms of Hilton--Hopf invariants instead of James--Hopf invariants. We use a restatement from the book of Baues \cite{baues}. Since his sign convention may be different from Whitehead \cite{whitehead}, we do not specify the signs.

\begin{prp}[{\cite[Proposition II.3.4]{baues}}]
\label{prp:barcus_barratt}
For homogeneous elements $\alpha,\beta\in\pi_*(X)$ and $\zeta\in\pi_*(S^{|\beta|})$,
\[
  [\alpha,\beta\circ\zeta]=\sum_{n\geq 1}\pm[[\dots[\alpha,\underbrace{\beta],
  \dots],\beta}_{n\text{ times}}]\circ E^{|\alpha|-1}\gamma_n(\zeta).
\]  
\end{prp}

Note that the sum is finite since $\gamma_n(\zeta)=0$ for $n> |\zeta|$. A special case of this formula was obtained by Barratt and Hilton as an application of Blakers--Massey relative Whitehead products \cite[\S3]{blakers_massey}.

\begin{crl}[{\cite[Theorem 3.24]{barratt_hilton}}]
\label{crl:BB-corollary}
Let $\alpha,\beta\in\pi_2(X).$ Then we have
\[
  [\alpha,\beta\circ\eta_2]=[\alpha,\beta]\circ\eta_3-[[\alpha,\beta],\beta]\in\pi_4(X).
\]  
\end{crl}

\begin{proof}
Proposition \ref{prp:barcus_barratt} gives
\[
  [\alpha,\beta\circ\eta_2]=[\alpha,\beta]\circ\eta_3\pm[[\alpha,\beta],\beta]\circ E\gamma_2(\eta_2),
\]  
since $E\gamma_n(\eta_2)=0$ for $n\ge 3$. Furthermore, $E\gamma_2(\eta_2)=\pm H_0(\eta_2)\iota_4=\pm\iota_4\in\pi_4(S^4)$ by Proposition~\ref{prp:james_properties}~(4).
We therefore obtain the required identity up to a sign in the last summand.
To identify the sign, we observe that
\[
  2[\alpha,\beta\circ\eta_2]=[\alpha,\beta\circ[\iota_2,\iota_2]]=[\alpha,[\beta,\beta]]=-2[[\alpha,\beta],\beta],
\]
where the first identity uses the well-known formula $[\iota_2,\iota_2]=2\eta_2$, the second identity follows by 
Proposition~\ref{prp:composition_properties}~(5), and the third is the Jacobi identity (Proposition~\ref{prp:whitehead_properties}~(3)). The identity above must be obtained from the required identity by multiplication by~$2$ (as $\eta_3$ has order~$2$), so the sign is minus.
\end{proof}

\subsection{Composition with iterated Hopf maps} In what follows, we denote by $\eta$ the suspended Hopf map $S^{n+1}\to S^n$, $n\ge 2$, as well as the corresponding element in $\pi_{n+1}(S^n)$. We also consider the $k$-fold iteration $\eta^k\colon S^{n+k}\to S^n$, $k\ge0$. The properties of the composition operations $(-)\circ\eta^k\colon\pi_n(X)\to \pi_{n+k}(X)$ are summarised next.

\begin{prp}
\label{prp:properties of hopf}
Let $X$ be a pointed space and let $\alpha,\beta\in\pi_*(X)$ be homogeneous elements.

If $|\alpha|\geq 3$, then:
\begin{enumerate}
    \item $[\beta,\alpha\circ\eta]=[\beta,\alpha]\circ\eta$;
    \item $(\alpha+\alpha')\circ\eta=\alpha\circ\eta+\alpha'\circ\eta$;
    \item $(n\alpha)\circ\eta =n(\alpha\circ\eta)$ for $n\in\ZZ$;
    \item $2(\alpha\circ\eta)=0$;
    \item $\alpha\circ\eta^4=0$.
\end{enumerate}
If $|\alpha|=2$, then:
\begin{enumerate}
    \item[($1'$)] $[\beta,\alpha\circ\eta]=[\beta,\alpha]\circ\eta\pm [[\beta,\alpha],\alpha]$;
    \item[($2'$)] $(\alpha+\alpha')\circ\eta=\alpha\circ\eta+\alpha'\circ\eta+[\alpha,\alpha']$;
    \item[($3'$)] $(n\alpha)\circ\eta=n(\alpha\circ\eta)+\binom{n}{2}[\alpha,\alpha]$ for $n\in\ZZ$,
    
    in particular, $(-\alpha)\circ\eta=-(\alpha\circ\eta)+[\alpha,\alpha]$;
    \item[($4'$)] $2(\alpha\circ\eta)=[\alpha,\alpha]$;
    \item[($5'$)] $\alpha\circ\eta^5=0$.
\end{enumerate}
\end{prp}

\begin{proof}
If $|\alpha|\geq 3$, then $\eta=E^{|\alpha|-2}\eta_2\in\pi_{|\alpha|+1}(S^{|\alpha|})$ is a suspension element; hence, properties (1) and (2) follow from Proposition~\ref{prp:composition_properties}. Property (3) follows from (2) by induction. Also, $2\eta_3=0$ and $\eta^4_3=0$ by~\cite[(5.3) and Proposition 5.8]{toda}; hence, $2(\alpha\circ\eta)=\alpha\circ(2\eta)=0$ and $\alpha\circ\eta^4=0$, which proves (4) and (5).

Property ($1'$) is Proposition~\ref{crl:BB-corollary}. For ($2'$), see, for example, \cite[Ch. XI, (1.16) and (4.4)]{whitehead}.
Property ($3'$) follows from ($2'$) by induction, and ($4'$) holds since $2(\alpha\circ\eta)=\alpha\circ(2\eta)=\alpha\circ[\id,\id]=[\alpha,\alpha]$. Finally, ($5'$) follows from (5), since $\alpha\circ\eta^5=(\alpha\circ\eta)\circ\eta^4=0$.
\end{proof}

\subsection{The case when $\ZK$ is a wedge of spheres}
\label{subsec:mac-wedge-of-spheres}

Let $\K$ be a simplicial complex on $[m]$ such that $\ZK\simeq\bigvee_{x\in X}S^{n_x}$ is homotopy equivalent to a wedge of simply connected spheres, $x\in\pi_{n_x}(\ZK)$. Then there is an additive isomorphism
\[
  \pi_*(\ZK)\cong\bigoplus_{b\in\BC(X)}w_b\circ\pi_*(S^{|b|})
\]  
by the Hilton--Milnor theorem.

The homotopy fibration~\eqref{eq:zk-dj fibration} gives an extension of quasi-Lie algebras
\[
  0\to \pi_*(\ZK)\to\pi_*(\DJ(\K))\to\ZZ\langle t_1,\dots,t_m\rangle\to 0
\]
and an isomorphism of abelian groups
\begin{equation}\label{piDJdec}
  \pi_*(\DJ(\K))\cong\ZZ\langle t_1,\dots,t_m\rangle\oplus\bigoplus_{b\in\BC(X)}w_b\circ\pi_*(S^{|b|}).
\end{equation}
The nontriviality of the extension above is expressed in the Whitehead products $[t_i,x]$ for $x\in X$. (Note that $n_x\geq 2$, so $[t_i,x]\in\pi_{n_x+1}(\DJ(\K))\cong\pi_{n_x+1}(\ZK)$.) Identifying the elements $[t_i,x]$ is the most important part of the description of the quasi-Lie algebra structure of $\pi_*(\DJ(\K))$. 

If $\K$ is a flag complex then $\ZK$ is a wedge of spheres if and only if $\K^1$ is a chordal graph. In this case, the wedge summands correspond to GPTW elements, i.\,e., $X=GPTW$, and the description of the elements $[t_i,x]$, $x\in GPTW$ is given in Section~\ref{sec:computations of brackets} as part of the proof of~Theorem \ref{thm:iterated whitehead product to GPTW}. 

Once the elements $[t_i,x]\in\pi_{n_x+1}(\ZK)$ are known, there is the following algorithmic description of the Whitehead product of any two elements in $\pi_*(\DJ(\K))$. There are Whitehead products of five types in~\eqref{piDJdec}:

\begin{itemize}
\item[$(1)$] $[t_i,t_j]$, $i,j\in [m]$;
\item[$(2)$] $[w_b,w_{b'}]$, $b,b'\in\BC(X)$;
\item[$(2')$] $[w_b\circ\alpha,w_{b'}\circ\alpha']$,   
  $\alpha\in\pi_*(S^{|b|})$, $\alpha'\in\pi_*(S^{|b'|});$
\item[$(3)$] $[t_i,w_b]$, $i\in[m]$, $b\in\BC(X)$;
\item[$(3')$] $[t_i,w_b\circ\alpha]$, $\alpha\in\pi_*(S^{|b|})$.
\end{itemize}

Type (1): $[t_i,t_j]$ is zero if $\{i,j\}\in\K$ and a GPTW element if $\{i>j\}\in\K$.

Types (2) and $(2')$ are computed in the homotopy groups of the wedge of spheres $\ZK\simeq\bigvee_{x\in X}S^{n_x}$, which is a classical computation. More explicitly, by Corollary~\ref{crl:lie_polynomial_normal_form}, an element of type (2) is a linear combination of elements of the form $w_{b''}$, $[w_{b''},w_{b''}]=w_{b''}\circ [\id,\id]$ and $[w_{b''},[w_{b''},w_{b''}]]=w_{b''}\circ[\id,[\id,\id]]$ for $b''\in\BC(X)$. Type $(2')$ is reduced to type (2) by repeated use of Propositions \ref{prp:barcus_barratt} and \ref{prp:whitehead_properties}.

Type (3): induction on $|b|$. For the base case, $b=x\in X$, and $[t_i,x]$ is known by assumption. For the inductive step, we have $b=[b',b'']$, $b',b''\in\BC(X)$. Then $[t_i,w_b]=[t_i,[w_{b'},w_{b''}]]=\pm [[t_i,w_{b'}],w_{b''}]\pm [[t_i,w_{b''}],w_{b'}]$. A decomposition $[t_i,w_{b'}]=\sum_sw_{b_s}\circ \alpha_{s}$ is known by inductive assumption, and then $[w_{b_s}\circ\alpha_{s},w_{b''}]$ is an expression of type $(2')$. The computation of $[[t_i,w_{b''}],w_{b'}]$ is similar.

Type $(3')$: by Proposition \ref{prp:barcus_barratt}, 
\[
  [t_i,w_b\circ\alpha]=\sum_{k\geq 1}[\dots[[t_i,\underbrace{w_b],w_b]\dots,w_b}_{k\text{ times}}]\circ E\gamma_k(\alpha).
\]  
The summand with $k=1$ is of type (3), while the other summands  are computed by induction on $k$, where the inductive step is a computation of type $(2')$.

\section{Composing iterated Whitehead products with Hopf elements}
\label{sec:computations of brackets}
In this section we develop the commutator calculus needed to approach Problem~\ref{prb:main1}. Namely, we prove that any bracket formed from $t_1,\ldots,t_m\in\pi_2(\DJ(\K))$ can be expressed via GPTW and Hopf elements.

\begin{thm}
\label{thm:iterated whitehead product to GPTW}
Let $w\in\pi_{|w|}(\DJ(\K))$ be an iterated Whitehead product of $t_1,\dots,t_m$ for $|w|>2$. Then
\[
  w = p_0+p_1\circ\eta+p_2\circ\eta^2+p_3\circ\eta^3,
\]
where $p_i$ is a Lie polynomial in GPTW elements, $0\le i\le 3$. 

Further, if no $t_i$ appears in $w$ twice, then $w=p_{0}$ is a Lie polynomial in GPTW elements.
\end{thm}

The proof will provide an algorithmic expression of commutators. It will depend on a series of lemmas that give explicit expressions for iterated Whitehead products of the canonical elements $t_1,\ldots,t_m\in\pi_2(\DJ(\K))$ with repeating indices. The most basic of commutator expressions is of special interest.

Let $\K=\{\varnothing,\{1\},\{2\}\}$ be a two point complex. Then the homotopy fibration~\eqref{eq:zk-dj fibration} becomes
\[
  S^3\to \CC P^\infty\vee\CC P^\infty\to\CC P^\infty\times\CC P^\infty
\]  
and the inclusion $\ZK\to\DJ(\K)$ identifies with the Whitehead product $[t_1,t_2]\colon S^3\to\CC P^\infty\vee\CC P^\infty$. We therefore have 
$\pi_*(\CC P^\infty\vee\CC P^\infty)\cong \ZZ\langle t_1, t_2\rangle
  \oplus[t_1,t_2]\circ\pi_*(S^3)$. 

\begin{prp}
\label{prp:iij}
Let $t_1,t_2\in\pi_2(\CC P^\infty\vee\CC P^\infty)$ be the standard generators. Then 
\[
  [t_1,[t_1,t_2]]=[t_1,t_2]\circ\eta\qquad\text{in}\;\;
  \pi_4(\CC P^\infty\vee\CC P^\infty)\cong\mathbb Z_2.
\]
For an arbitrary simplicial complex $\K$ on $[m]$, we have 
\[
  [t_i,[t_i,t_j]]=[t_i,t_j]\circ\eta\qquad\text{in}\;\;\pi_4(\DJ(\K))
\]
for the standard generators $t_1,\dots,t_m\in\pi_2(\DJ(\K))$ and any $i,j\in[m]$.
\end{prp}

\begin{proof} 
Since $\pi_{3}(\CC P^{\infty})=0$, we obtain 
$t_1\circ\eta=0\in\pi_3(\CC P^\infty\vee \CC P^\infty)$, hence $[t_1,[t_1,t_2]]=[t_1,t_2]\circ\eta$ by the Barratt--Hilton formula (Corollary~\ref{crl:BB-corollary}). 
In the more general case for $\pi_{4}(\DJ(\K))$, the result follows by considering the map of polyhedral products 
$\CC P^{\infty}\vee\CC P^{\infty}\longrightarrow\DJ(\K)$ induced by including the vertices $i$ and $j$ into $\K$.
\end{proof}

\begin{rmk}
The nontriviality of $[t_1,[t_1,t_2]]$ was established by Kallel~\cite[Lemma 8.7]{kallel} using the Postnikov tower of $\CC P^\infty\vee\CC P^\infty$. 
Since $[t_1,t_2]\circ\eta$ is a generator of $\pi_4(\CC P^\infty\vee\CC P^\infty)=\pi_4(S^3)\cong\ZZ_2$, Kallel's result gives an alternative proof of Proposition~\ref{prp:iij}.
\end{rmk}

For the two point complex $\K$, the only GPTW element is $[t_2,t_1]$, so Proposition~\ref{prp:iij} together with the algorithm in~\S\ref{subsec:mac-wedge-of-spheres} can be used to describe the Whitehead product in $\DJ(\K)=\CC P^\infty\vee\CC P^\infty$ completely.

\begin{prp}\label{qla2p}
The Whitehead product in 
\[
  \pi_*(\CC P^\infty\vee\CC P^\infty)\cong \ZZ\langle t_1, t_2\rangle
  \oplus[t_2,t_1]\circ\pi_*(S^3)
\]
is described by the relations
\[
  [t_i,t_i]=0,\quad
  \bigl[[t_2,t_1]\circ\alpha,[t_2,t_1]\circ\alpha'\bigr]=0,\quad
  \bigl[t_i,[t_2,t_1]\circ\alpha\bigr]=[t_2,t_1]\circ(\eta\circ\Sigma\alpha)
\]
for $i=1,2$ and $\alpha,\alpha'\in\pi_*(S^3)$.
\end{prp}
\begin{proof}
Both $\CC P^\infty$ and $S^3$ are H-spaces, so Whitehead products in their homotopy groups are trivial. This proves the first two relations. To prove the third one we denote $x=[t_2,t_1]$ and use Proposition~\ref{prp:barcus_barratt} to write the expansion,
\[
  [t_i,x\circ\alpha]=[t_i,x]\circ\Sigma\alpha+[[t_i,x],x]\circ \beta_2+[[[t_i,x],x],x]\circ\beta_3+\cdots,
\]  
where $\beta_k$ are some elements in the homotopy groups of spheres. On the other hand, $[t_i,x]=x\circ\eta$ by Proposition~\ref{prp:iij}, so $[[t_i,x],x]=[x\circ\eta,x\circ\id]=0$ by the second relation. It follows that all summands in the expansion above vanish except for the first one, and hence  $[t_i,x\circ\alpha]=[t_i,x]\circ\Sigma\alpha=x\circ\eta\circ\Sigma\alpha$.
\end{proof}

Now we give a sequence of lemmas leading to the proof of Theorem~\ref{thm:iterated whitehead product to GPTW}.

\begin{lmm}[{cf. \cite[Lemma 3.8]{vv}}]
\label{lmm:adding_letter}
Consider a nested commutator $$x=[t_{j_1},[\dots[t_{j_k},t_i]\dots]]\in\pi_*(\DJ(\K)),$$ where $j_1,\dots,j_k,i\in[m]$ and $k\geq 1.$ Then $[t_{j_1},x]=x\circ\eta.$
\end{lmm}
\begin{proof}
In this proof, we write $t_1:=t_{j_1}$ and $t_2:=t_{j_2}$ and use induction on~$k$. The base case $k=1$ is Proposition~\ref{prp:iij}. For the inductive step, we write $x=[t_{1},[t_{2},y]]$, where $y=t_i$ for $k=2$ and $y$ is a nested commutator for $k>2$.

We have $[t_{1},[t_{1},y]]=[t_{1},y]\circ\eta$ by the inductive assumption, and $[t_1,[t_1,t_2]]=[t_1,t_2]\circ\eta$ by Proposition \ref{prp:iij}. Using the identities for the Whitehead product (Proposition~\ref{prp:whitehead_properties}) iteratively, we calculate
\[
\begin{split}
  [t_1,x] & =[t_1,[t_1,[t_2,y]]]=[t_1,[[t_2,y],t_1]]  =-[t_1,[[t_1,t_2],y]]-[t_1,[[y,t_1],t_2]] \\
  & =-[[[t_1,t_2],y],t_1]-[[[y,t_1],t_2],t_1] \\
  & =[[t_1,[t_1,t_2]],y]+(-1)^{|y|}[[y,t_1],[t_1,t_2]]+
  [[t_2,t_1],[y,t_1]]+[[t_1,[y,t_1]],t_2]\\
  & =[[t_1,[t_1,t_2]],y]+[t_2,[t_1,[t_1,y]]]  =[[t_1,t_2]\circ\eta,y]+[t_2,[t_1,y]\circ\eta]\\
  & =\bigl([[t_1,t_2],y]+[t_2,[t_1,y]]\bigr)\circ\eta=x\circ\eta 
\end{split}
\]
where the second last identity uses Proposition~\ref{prp:properties of hopf}~(1) and the sign in the last two identities is irrelevant as $\eta$ is a suspension and has order $2$.
\end{proof}

Given an ordered subset $J=\{j_1<\dots<j_k\}\subset[m]$ and an element $y\in\pi_N(\DJ(\K))$, we write 
\[
  c(J,y):=[t_{j_1},[t_{j_2},\dots[t_{j_k},y]\dots]]\in\pi_{N+k}(\DJ(\K)).
\]  
Note that $c(J,t_j)=c(J,j;\underline t)$ in the notation of \S\ref{notiwp}.

For $A,B\subset[m]$, consider the Koszul sign 
\[
  \theta(A,B):=\#\{(i,j)\in A\times B:i>j\}.
\]

\begin{lmm}\label{lmm:c(A,[x,y])}
Given $P\subset[m]$ and $x,y\in\pi_*(\DJ(\K))$, we have 
\[
  (-1)^{|P|}c(P,[x,y])=\sum_{P=A\sqcup B}(-1)^{\theta(A,B)+|x|\cdot|B|}\bigl[c(A,x),c(B,y)\bigr].
\]  
In particular,
$
  (-1)^{|P|}c(P,[t_i,y])=\sum_{P=A\sqcup B}(-1)^{\theta(A,B)}\bigl[c(A,t_i),c(B,y)\bigr].
$
\end{lmm}

\begin{proof}
A similar identity in a quasi-Lie algebra was obtained in~\cite[Lemma~C.3]{vylegzhanin}; we reproduce the argument taking into account the sign difference between the Whitehead product and the quasi-Lie bracket. We use the identity $-[t_i,[\beta,\gamma]]=[[t_i,\beta],\gamma]+(-1)^{|\beta|}[\beta,[t_i,\gamma]]$, which follows from Proposition \ref{prp:whitehead_properties}, and argue by induction on $|P|$. The base case $P=\varnothing$ is trivial. For the inductive step, denote $P=\{i\}\sqcup P'$, $i=\min(P)$.  Then
\begin{align*}
(-1)^{|P|}c(P,[x,y])&=(-1)^{|P'|+1}[t_i,c(P',[x,y])]\\
&=\sum_{P'=A'\sqcup B'}(-1)^{\theta(A',B')+|x|\cdot |B'|}\cdot (-[t_i,[c(A',x),c(B',y)]])\\
&=\sum_{P'=A'\sqcup B'}(-1)^{\theta(\{i\}\sqcup A',B')+|x|\cdot |B'|}[c(\{i\}\sqcup A',x),c(B',y)]\\
&+\sum_{P'=A'\sqcup B'}(-1)^{\theta(A',\{i\}\sqcup B')+|x|\cdot(1+|B'|)}[c(A',x),c(\{i\}\sqcup B',y)]\\
&=\sum_{\begin{smallmatrix}
P=A\sqcup B,~i\in A
\end{smallmatrix}}(-1)^{\theta(A,B)+|x|\cdot |B|}[c(A,x),c(B,y)]\\&+\sum_{\begin{smallmatrix}
P=A\sqcup B,~i\in B
\end{smallmatrix}}(-1)^{\theta(A,B)+|x|\cdot |B|}[c(A,x),c(B,y)].\qedhere
\end{align*}
\end{proof}

\begin{lmm}
\label{lmm:[j,c(Q,j)]}
Let $j\in [m],$ $Q\subset[m]$, $Q\neq\varnothing$. Then
\[
\bigl[t_j,c(Q,t_j)\bigr]=c(Q,t_j)\circ\eta-\sum_{\begin{smallmatrix}
  Q=A\sqcup B,\\
  A\neq\varnothing,\\
  \max(Q)\in B
  \end{smallmatrix}}
  (-1)^{\theta(A,B)}\bigl[c(A,t_j),c(B,t_j)\bigr].
\]
\end{lmm}

\begin{proof}
Denote $n=\max(Q)$ and $P=Q\setminus\{n\}$. We have
\begin{align*}
    c(Q,t_j)\circ\eta &=
  (-1)^{|P|}c(Q,t_j)\circ\eta = 
  (-1)^{|P|}c(P,[t_n,t_j]\circ\eta)
  =(-1)^{|P|}c\bigl(P,[t_j,[t_n,t_j]]\bigr)
  \\&=\sum_{P=A\sqcup B'}\!\!\!(-1)^{\theta(A,B')}
  \bigl[c(A,t_j),c(B',[t_n,t_j])\bigr]
  \\ &=\bigl[t_j,c(P,[t_n,t_j])\bigr]+
  \sum_{\begin{smallmatrix} P=A\sqcup B',\\
  A\neq\varnothing\end{smallmatrix}}
  (-1)^{\theta(A,B')}\bigl[c(A,t_j),c(B'\sqcup\{n\},t_j)\bigr]
  \\ &=[t_j,c(Q,t_j)]+\sum_{\begin{smallmatrix} Q=A\sqcup B,\\
  A\neq\varnothing,~n\in B\end{smallmatrix}}
  (-1)^{\theta(A,B)}\bigl[c(A,t_j),c(B,t_j)\bigr],
\end{align*}
where the first identity follows since $\eta$ is of order two, the second is by Proposition~\ref{prp:properties of hopf}~(1), the third is by Proposition~\ref{prp:iij} and the fourth is by Lemma~\ref{lmm:c(A,[x,y])}.
\end{proof}

We use the notation
\[
  I_{<i}:=\{k\in I:k<i\},\quad I_{\geq i}=\{k\in I:k\geq i\}.
\]

\begin{lmm}
\label{prp:comm_with_gptw_formula}
Suppose $I\subset[m]$, $i,j\in[m]$, $j\notin I$. Then
\[\begin{split}
  [t_i,c(I,t_j)]= & -\sum_{I_{<i}=A\sqcup B,\:
  A\neq\varnothing}
  (-1)^{\theta(A,B)}\bigl[c(A,t_i),c(I\setminus A,t_j)\bigr] \\ 
  & \qquad +\begin{cases}
(-1)^{|I_{<i}|}c(I\sqcup\{t_i\},t_j)& \text{if}\quad i\notin I;\\
c(I,t_j)\circ\eta& \text{if}\quad i\in I.
\end{cases}
\end{split}\]
\end{lmm}

\begin{proof}
Let $P=I_{<i}$, $Q=I_{\geq i}$. By Lemma~\ref{lmm:c(A,[x,y])},
\begin{align*}
 &(-1)^{|P|}c\bigl(P,[t_i,c(Q,t_j)]\bigr) =
  \sum_{P=A\sqcup B}(-1)^{\theta(A,B)}\bigl[c(A,t_i),c(B,c(Q,t_j))\bigr]
  \\ &=\bigl[t_i,c(P,c(Q,t_j))\bigr]  \qquad +\sum_{P=A\sqcup B,\:A\neq\varnothing}(-1)^{\theta(A,B)}\bigl[c(A,t_i),c(B,c(Q,t_j))\bigr]
  \\ &=[t_i,c(I,t_j)]+\!\!\!\!\sum_{A\subset P,\:A\neq\varnothing}(-1)^{\theta(A,B)}\bigl[c(A,t_i),c(I\setminus A,t_j)\bigr].
\end{align*}
Now, if $i\notin I$, then $c(P,[t_i,c(Q,t_j)])=c(I\sqcup\{i\},t_j)$. If  $i\in I$, then
\[
\begin{split}
  c\bigl(P,[t_i,c(Q,t_j)]\bigr) & 
  =c\bigl(P,[t_i,[t_i,c(Q\setminus i,t_j)]]\bigr) \\ 
  & =c\bigl(P,[t_i,c(Q\setminus i,t_j)]\circ\eta\bigr)
  =c(I,t_j)\circ\eta.\qedhere
\end{split}
\]
\end{proof}

\begin{proof}[Proof of Theorem \ref{thm:iterated whitehead product to GPTW}]
We argue by induction on $|w|$. The base case is $|w|=3$, in which case $w=[t_i,t_j]=[t_j,t_i]$. Then $w=0$ if $\{i,j\}\in\K$ and $w\in GPTW$ if $\{i>j\}\notin\K$. The induction step has several cases.

Case 1: $w$ is an iterated commutator on $t_1,\dots,t_m$ without repeating indices. This case is done in~\cite[Theorem~4.3]{gptw}. Namely, by repeated use of the Jacobi identity, we can assume that $w$ is a Lie polynomial on nested commutators, i.\,e. commutators of the form $[t_{i_1},[\dots[t_{i_s},t_j]\dots]]$, $s\geq 1$. We then use the Jacobi identity to reorder $i_1,\dots,i_s,j$ so that $i_s=\max\{i_1,\dots,i_s,j\}$ (see the proof of \cite[Lemma 3.1]{vv} for an explicit algorithm). Finally, we express every such nested commutator as a Lie polynomial in GPTW elements using \cite[Algorithm~5.4]{vylegzhanin}.

Case 2: $w=[t_i,y]$, where $y$ is a GPTW element. Then $y=c(I,t_j)$ for some $I\subset[m]$ and $j\notin I$. If $i= j$, we use Lemma~\ref{lmm:[j,c(Q,j)]} to express $[t_i,y]=[t_j,c(I,t_j)]$ through iterated commutators without repeating indices, to which Case 1 applies.  If $i\ne j$, then we use Lemma~\ref{prp:comm_with_gptw_formula}. Whenever a Hopf element appears inside an iterated commutator, we use identities (1)--(5) from Proposition~\ref{prp:properties of hopf} to reduce it to the form $p\circ\eta^k$ where $k\le 3$.

Case 3: $w=[t_i,y]$, where $y$ is a Lie polynomial on GPTW elements. If $y$ is a GPTW element, Case 2 applies; otherwise $y$ is decomposable, so we use the Jacobi identity $[t_i,[y_1,y_2]]=[[t_i,y_1],y_2]\pm [y_1,[t_i,y_2]]$ to reduce the computation to Case~2.

The general case: $w=[w_1,w_2]$, where $w_1$ and $w_2$ are iterated Whitehead products of smaller length. Applying the Jacobi identity iteratively, we express $w$ as a linear combination of elements of the form $[t_i,y]$, where $y$ is an iterated Whitehead product of $t_1,\dots,t_m$. By the inductive assumption applied to $y$, we can assume that $y$ is a Lie polynomial in GPTW elements, and Case~3 applies.
\end{proof}

The following expressions for iterated Whitehead products with repeating indices will be used in Section~\ref{sec:homotopy groups of dj}.

\begin{lmm}
\label{lmm:repeats in a nested commutator}
Let $J\subset[m]$ and $j\in J\setminus\max(J)$. Then
\[
  c(J,t_j)=c(J\setminus j,t_j)\circ\eta-\sum_{\begin{smallmatrix}
  J\setminus j=A\sqcup B,\\\max(A)>j,\\ \max(J)\in B\end{smallmatrix}}
  (-1)^{\theta(A,B)+|A_{<j}|}\bigl[c(A,t_j),c(B,t_j)\bigr].
\]  
\end{lmm}

\begin{proof}
We write $I=P\sqcup\{j\}\sqcup Q$, where $\max(P)<j<\min(Q)$. Then
\[\begin{split}
  c(J,t_j) & =c(P,[t_j,c(Q,t_j)]) \\ 
  & =c(P,c(Q,t_j)\circ\eta)-\!\!\!\sum_{\begin{smallmatrix}
  Q=A'\sqcup B',\\A'\neq\varnothing,\\\max(Q)\in B'\end{smallmatrix}}
  \!\!\!(-1)^{\theta(A',B')} c(P,[c(A',t_j),c(B',t_j)])
  \\
  &=c(J\setminus j,t_j)\circ\eta-
  \sum_{\begin{smallmatrix}
  Q=A'\sqcup B',\\A'\neq\varnothing,\\\max(Q)\in B'\end{smallmatrix}}
  \sum_{P=A''\sqcup B''}
(-1)^{\epsilon} [c(A''\sqcup A',t_j),c(B''\sqcup B',t_j)],
\end{split}\]
where the second identity is by Lemma~\ref{lmm:[j,c(Q,j)]} and the third is by Lemma~\ref{lmm:c(A,[x,y])}. The sign is given by
\[
  \epsilon=\theta(A',B')+|P|+\theta(A'',B'')+(|A'|+1)\cdot|B''|
  %=\theta(A'\sqcup A'',B'\sqcup B'')+|A''|
  =\theta(A,B)+|A_{<j}|,
\]
where $A:=A'\sqcup A''$, $B:=B'\sqcup B''$, so that $A_{<j}=A''$.
\end{proof}

For a multi-index $\alpha\in\Zm$ and an index $j\in[m]$ we use the notation $c(\alpha,t_j)=c(\alpha,j;\underline t)\colon S^{|\alpha|+2}\to\DJ(\K)$, see~\eqref{cprod}.

\begin{lmm}
\label{lmm:computing c(alpha-j,j)}
Let $\alpha\in\Zm$, $J=\supp\alpha$ and $j\in J\setminus\max(J)$. Then 
\[
  c(\alpha-j,t_j)=c(J\setminus j,t_j)\circ\eta^{|\alpha|-|J|}+
  \sum_{\begin{smallmatrix}J\setminus j=A\sqcup B,\\
          \max(A)>j,\\
          \max(J)\in B
        \end{smallmatrix}}
  \pm [c(A,t_j),c(B,t_j)]\circ\eta^{|\alpha|-|J|-1},
\]
where the second summand on the right hand side is zero if $\alpha_j=1$.
\end{lmm}

\begin{proof}
Denote $\beta=\alpha-j$. By repeated use of Lemma \ref{lmm:adding_letter} and Proposition \ref{prp:whitehead_properties}, 
\[
  c(\beta,t_j)=c(\supp\beta,t_j)\circ\eta^{|\beta|-|\supp\beta|}=\begin{cases}
  c(J\setminus j,t_j)\circ \eta^{|\alpha|-|J|} &\text{if }\alpha_j=1,\\
  c(J,t_j)\circ\eta^{|\alpha|-|J|-1} &\text{if }\alpha_j>1.
  \end{cases}
\]
In the first case, we are done. In the second case we plug in the expression for $c(J,t_j)$ from Lemma \ref{lmm:repeats in a nested commutator} to obtain the required formula. 
\end{proof}

\begin{rmk}
The nested commutators $c(A,t_j)$ and $c(B,t_j)$ in Lemma~\ref{lmm:computing c(alpha-j,j)} satisfy $\max(A),\max(B)>j$ and hence are GPTW elements if $\K=\L$ is the disjoint union of $m$ points. In general, $c(A,t_j)$ and $c(B,t_j)$ are not necessarily GPTW elements, and expressing $c(\alpha-j,t_j)$ through GPTW elements requires additional calculations, see~\cite[Algorithm~5.4]{vylegzhanin}.
\end{rmk}

\section{Relations between iterated Whitehead products}
\label{sec:identities}

Iterated Whitehead products and, in particular, GPTW elements in $\pi_*(\DJ(\K))$ are subject to a set of Lie relations, which are described in this section. 

When $\K$ is a flag complex, the loop homology algebra $H_*(\Omega\ZK;\k)$ is minimally generated by the adjoints of GPTW elements~\cite[Theorem~4.3]{gptw}. A minimal set of relations between the adjoints of GPTW elements in $H_*(\Omega\ZK;\k)$ was described in \cite[Theorem~1.1]{vylegzhanin}. These relations correspond to generators $\kappa\in H_1(\K_J;\k)$ for $J\subset[m]$ and have the form $r_\kappa=0$, where $r_\kappa$ is a quadratic Lie polynomial in iterated Whitehead products $c(A,t_i)$.
Here we prove the stronger result that the relations $r_\kappa=0$ hold already in $\pi_*(\DJ(\K))$. 

For ordered subsets $A=\{a_1<\cdots<a_k\}$ an $B=\{b_1<\cdots<b_l\}$ of $[m]$ and $x\in\pi_*(\DJ(\K))$, we recall the notation from the previous section:
\[
  c(A,x)=[t_{a_1},[\dots[t_{a_k},x]\dots]],\qquad  
  \theta(A,B)=\#\{(a,b)\in A\times B\colon a>b\}.
\]  
 
\begin{thm}\label{thmminrel}
Let $\K$ be a simplicial complex on $[m]$. For any $J\subset[m]$ and any simplicial cycle $\kappa=\sum_{\{i<j\}\in\K_J}\lambda_{ij}\{i,j\}\in C_1(\K_J;\ZZ)$, there is the following relation in $\pi_*(\DJ(\K))$:
\begin{equation}\label{minrel}
  \sum_{\{i<j\}\in\K_J}\lambda_{ij}\sum_{\begin{smallmatrix}
  J\setminus\{i,j\}=A\sqcup B:\\
  \max(A)>i,\;\max(B)>j \end{smallmatrix}}
  (-1)^{\theta(A,B)}\bigl[c(A,t_i),c(B,t_j)\bigr]=0
\end{equation}
\end{thm}

\begin{rmk}
For any topological space~$X$, the same relations hold between elements $t_1,\dots,t_m\in \pi_*(X)$ of even dimension that satisfy $[t_i,t_j]=0$ for any $\{i,j\}\in\K$. Algebraically, these relations are corollaries of the relations $[t_i,t_j]=0$ for $\{i,j\}\in\K$ in a quasi-Lie algebra generated by  elements $t_1,\ldots,t_m$ of odd degree. 

It is shown in~\cite[Theorem~1.1]{vylegzhanin} that the relations~\eqref{minrel} form a minimal set of relations between the adjoints of GPTW elements in the loop homology algebra $H_*(\Omega\ZK;\k)$, for any principal ideal domain~$\k$. This gives a minimal presentation for $H_*(\Omega\ZK;\k)$ when $\K$ is a flag complex: it is generated by the GPTW elements modulo relations~\eqref{minrel} corresponding to additive generators of $H_1(\K_J;\k)$, $J\subset[m]$. To express each $c(A,t_j)$ in~\eqref{minrel} through GPTW elements one needs to use \cite[Algorithm~5.4]{vylegzhanin}.
\end{rmk}

We shall need the following version of identities from Lemma~\ref{lmm:c(A,[x,y])}.

\begin{lmm}\label{lmm:cJx2} 
Let $J\subset[m]$, $i,j\in J$ and $i<j$. Then
\begin{gather*}
  c(J\setminus i,t_i)=
  (-1)^{|J_{<j}|+1}
  \sum_{\begin{smallmatrix} J\setminus\{i,j\}=A\sqcup B:\\
  \max(B)<j\end{smallmatrix}}(-1)^{\theta(A,B)}\bigl[c(A,t_i),c(B,t_j)\bigr],\\
  c(J\setminus j,t_j)=
  (-1)^{|J_{<i}|}
  \sum_{\begin{smallmatrix}
  J\setminus\{i,j\}=A\sqcup B:\\
  \max(A)<i\end{smallmatrix}}(-1)^{\theta(A,B)}\bigl[c(A,t_i),c(B,t_j)\bigr].
\end{gather*}
\end{lmm}

\begin{proof}
Write $J=J_1\sqcup\{i\}\sqcup J_2\sqcup \{j\}\sqcup J_3$, where we let $J_1=\{t\in J\colon t<i\}$, $J_2=\{t\in J\colon i<t<j\}$ and $J_3=\{t\in J\colon t>j\}$. Then by Lemma~\ref{lmm:c(A,[x,y])} we have 
\[
\begin{split} 
  (-1)^{|J_{<j}|+1}
  c(J\setminus i,t_i) & =
  (-1)^{|J_1\sqcup J_2|}
  c\bigl(J_1\sqcup J_2,[t_j,c(J_3,t_i)]\bigr) \\ 
  & =(-1)^{|J_1\sqcup J_2|}c\bigl(J_1\sqcup J_2,[c(J_3,t_i),t_j]\bigr) \\
  & =\sum_{J_1\sqcup J_2=A'\sqcup B}
  (-1)^{\theta(A',B)+(|J_3|+2)|B|}
  \bigl[c(A'\sqcup J_3,t_i),c(B,t_j)\bigr] \\
  & =\sum_{J_1\sqcup J_2=A'\sqcup B}(-1)^{\theta(A'\sqcup J_3,B)}\bigl[c(A'\sqcup J_3,t_i),c(B,t_j)\bigr] \\ 
  & =\sum_{\begin{smallmatrix}
  J_1\sqcup J_2\sqcup J_3=A\sqcup B:\\
  J_3\subset A
  \end{smallmatrix}}(-1)^{\theta(A,B)}\bigl[c(A,t_i),c(B,t_j)\bigr],
\end{split}
\] 
\[
\begin{split}
  (-1)^{|J_{<i}|}c(J\setminus j,t_j) & =(-1)^{|J_1|}
  c\bigl(J_1,[t_i,c(J_2\sqcup J_3,t_j)]\bigr) \\ 
  & =\sum_{J_1=A\sqcup B'}(-1)^{\theta(A,B')}
  \bigl[c(A,t_i),c(B'\sqcup J_2\sqcup J_3,t_j)\bigr] \\
  & =\sum_{J_1=A\sqcup B'}(-1)^{\theta(A,B'\sqcup J_2\sqcup J_3)}
  \bigl[c(A,t_i),c(B'\sqcup J_2\sqcup J_3,t_j)\bigr] \\
  & =\sum_{\begin{smallmatrix} J_1\sqcup J_2\sqcup J_3=A\sqcup B:\\
  J_2\sqcup J_3\subset B \end{smallmatrix}}(-1)^{\theta(A,B)}\bigl[c(A,t_i),c(B,t_j)\bigr].
  \qedhere
\end{split}
\] 
\end{proof}

\begin{proof}[Proof of Theorem~\ref{thmminrel}]
We write the second sum in~\eqref{minrel} using the inclusion-exclusion formula:
\[
  \sum_{\begin{smallmatrix}
J\setminus\{i,j\}= A\sqcup B:\\
\max(A)>i,~\max(B)>j
\end{smallmatrix}}\!\!=\!\!\sum_{J\setminus\{i,j\}=A\sqcup B}\!-\!\sum_{\begin{smallmatrix}
J\setminus\{i,j\}=A\sqcup B:\\
\max(B)<j
\end{smallmatrix}}\!-\!\sum_{\begin{smallmatrix}
J\setminus\{i,j\}=A\sqcup B:\\
\max(A)<i
\end{smallmatrix}}\!+\!\sum_{\begin{smallmatrix}
J\setminus\{i,j\}=A\sqcup B:\\
\max(A)<i,~\max(B)<j
  \end{smallmatrix}}
\]
where the summands are $(-1)^{\theta(A,B)}[c(A,t_i),c(B,t_j)]$ and are omitted. The first sum on the right hand side equals $(-1)^{|J|}c(J\setminus\{i,j\},[t_i,t_j])$ by Lemma~\ref{lmm:c(A,[x,y])}, which is zero as $\{i<j\}\in\K$ in~\eqref{minrel}. The fourth sum is empty unless $j=\max(J)$, in which case the sum on the left hand side is empty. Hence, the fourth sum above can be omitted. We denote $n=\max(J)$ and plug in the expressions for the second and third sum from Lemma~\ref{lmm:cJx2} to obtain
\[\begin{split}
  \sum_{\{i<j\}\in\K_J}&\lambda_{ij}
  \sum_{\begin{smallmatrix}J\setminus\{i,j\}=A\sqcup B:\\
  \max(A)>i,~\max(B)>j\end{smallmatrix}}(-1)^{\theta(A,B)}\bigl[c(A,t_i),c(B,t_j)\bigr]\\
  &=\sum_{\{i<j\}\in\K_J,~j\ne n} \lambda_{ij} 
  \Bigl((-1)^{|J_{<j}|}c(J\setminus i,t_i)
  -(-1)^{|J_{<i}|}c(J\setminus j,t_j)\Bigr)\\
  & =\sum_{i\in J\setminus\{n\}}
  \Bigl(\sum_{j\colon \{i<j\}\in\K_J,~j\neq n}\lambda_{ij}
  -\sum_{j\colon \{j<i\}\in\K_J}\lambda_{ji}\Bigr)(-1)^{|J_{<j}|}c(J\setminus i,t_i)\\
  & =\sum_{i\in J\setminus\{n\}}
  \Bigl(\sum_{j\colon \{i<j\}\in\K_J}\lambda_{ij}
  -\sum_{j\colon \{j<i\}\in\K_J}\lambda_{ji}\Bigr)(-1)^{|J_{<j}|}c(J\setminus i,t_i)\\
  & \hspace{0.1\textwidth} 
  -\sum_{i\in J\setminus\{n\},~\{i<n\}\in\K_J}
  \lambda_{in}(-1)^{|J_{<n}|}c(J\setminus i,t_i).
\end{split}
\] 
The last sum is zero since $c(J\setminus i,t_i)=c(J\setminus\{i,n\},[t_{n},t_i])=0$ whenever $\{i<n\}\in\K$. The sum in brackets is also equal to zero, since $\sum_{\{i<j\}\in\K_J}\lambda_{ij}\{i,j\}$ is a simplicial cycle. Relation~\eqref{minrel} follows.
\end{proof}

\begin{exm}
Let $\K$ be a $5$-cycle (the boundary of a pentagon). The generator of $H_1(\K;\ZZ)$ is represented by the cycle
\[
  \kappa=\{1,2\}+\{2,3\}+\{3,4\}+\{4,5\}-\{1,5\}.
\]

For $\{i,j\}=\{4,5\}$ or $\{1,5\}$, the second sum in~\eqref{minrel} is empty, as $\max(B)>5$ cannot be satisfied.

For $\{i,j\}=\{3,4\}$, we must have $\max(B)=5$. Then $c(B,t_j)=0$ as $[t_5,t_4]=0$, and the second sum in~~\eqref{minrel} is zero.

For $\{i,j\}=\{2,3\}$, the only partitions of $[5]\setminus\{i,j\}=\{1,4,5\}$ giving nonzero summands are
$\{1,4,5\}=\{4\}\sqcup\{1,5\}$ and $\{1,4,5\}=\{1,4\}\sqcup\{5\}$. (For other partitions, at least one of the commutators $c(A,t_i)$ and $c(B,t_j)$ is zero. For instance, for $\{5\}\sqcup\{1,4\}$ we have $c(B,t_j)=c(\{1,4\},t_3)=0$.) 

For $\{i,j\}=\{1,2\}$, the only partitions giving nonzero summands are
$\{3,4,5\}=\{3\}\sqcup\{4,5\}$, $\{3,4,5\}=\{4\}\sqcup\{3,5\}$ and $\{3,4,5\}=\{3,4\}\sqcup\{5\}$.

The resulting relation~\eqref{minrel} has the form
\begin{multline*}
  (-1)^{\theta(\{3\},\{4,5\})}\bigl[c\bigl(\{3\},t_1\bigr),
  c\bigl(\{4,5\},t_2\bigr)\bigr]
  +(-1)^{\theta(\{4\},\{3,5\})}\bigl[c\bigl(\{4\},t_1\bigr),
  c\bigl(\{3,5\},t_2\bigr)\bigr]\\
  +(-1)^{\theta(\{3,4\},\{5\})}\bigl[c\bigl(\{3,4\},t_1\bigr),
  c\bigl(\{5\},t_2\bigr)\bigr]
  +(-1)^{\theta(\{4\},\{1,5\})}\bigl[c\bigl(\{4\},t_2\bigr),
  c\bigl(\{1,5\},t_3\bigr)\bigr]\\
  +(-1)^{\theta(\{1,4\},\{5\})}\bigl[c\bigl(\{1,4\},t_2\bigr),
  c\bigl(\{5\},t_3\bigr)\bigr]=0.
\end{multline*}
Here all commutators $c(A,t_i)$ and $c(B,t_j)$ are GPTW elements, with the exception of $c(\{1,4\},t_2)=[t_1,[t_4,t_2]]=-[t_2,[t_4,t_1]]$, where the latter is a GPTW element. Calculating the signs, we obtain the following single relation on the $10$ GPTW elements generating $\pi_*(\mathcal Z_\mathcal K)$:
\begin{multline*}
  \bigl[[t_3,t_1],[t_4,[t_5,t_2]]\bigr]
  -\bigl[[t_4,t_1],[t_3,[t_5,t_2]]\bigr]
  +\bigl[[t_3,[t_4,t_1]],[t_5,t_2]\bigr]\\
  -\bigl[[t_4,t_2],[t_1,[t_5,t_3]]\bigr]
  -\bigl[[t_2,[t_4,t_1]],[t_5,t_3]]\bigr]=0.
\end{multline*}
This implies that $\mathcal Z_\mathcal K$ is obtained from the wedge of $10$ spheres $\bigvee_{x\in GTPW} S^{|x|}=(S^3\vee S^4)^{\vee 5}$ corresponding to GTPW elements by attaching a single $7$-cell by the relation above. The result is a connected sum of products of spheres, $\mathcal Z_\mathcal K\cong (S^3\times S^4)^{\#5}$, in accordance with the description of~\cite{MacGavran}.
See~\cite[\S5.4]{vylegzhanin} for the discussion of the relation in $\pi_*(\mathcal Z_{\mathcal K})$ in the case when $\K$ is an $m$-cycle. 
\end{exm}

\section{A criterion for when $\pi_{\ast}(\DJ(\K))$ is generated by iterated Whitehead products}
\label{sec:homotopy groups of dj}

We consider the map given by the iterated Whitehead products corresponding to GPTW elements and its lift to $\ZK$:
\[
  g_\K\colon\bigvee_{x\in GPTW}S^{|x|}\to \DJ(\K),\qquad
  \widehat g_\K\colon \bigvee_{x\in GPTW}S^{|x|}\to\ZK.
\]  
Define the homomorphism
\begin{equation}\label{varPhi}
  \varPhi\colon \ZZ\langle t_1,\ldots,t_m\rangle \oplus \pi_*\Bigl(\bigvee_{x\in GPTW}S^{|x|}\Bigr)\to \pi_*(\DJ(\K))
\end{equation}
by mapping each $t_i$ to the corresponding canonical element of $\pi_2(\DJ(\K))$ and using the map induced by $g_\K$ on the wedge of spheres. 

We start with the following equivalent description of the $\Pi$-subalgebra $S(\K)$ of $\pi_*(\DJ(\K))$ featuring in  Problem~\ref{prb:main2}. 

\begin{prp}\label{SKequiv}
The following subgroups of $\pi_*(\DJ(\K))$ coincide:
\begin{itemize}
\item[(a)]
the subgroup $S(\K)$ generated by all elements of the form $w\circ \alpha$, where $w\in \pi_*(\DJ(\K))$ is an iterated commutator of $t_1,\dots,t_m$ and $\alpha\in\pi_*(S^{|w|})$; 

\item[(b)]
the image of $\pi_*((S^2)^{\vee m})$ under the canonical inclusion $(S^2)^{\vee m}\to\DJ(\K)$;

\item[(c)]
the image of the map $\varPhi$ given by~\eqref{varPhi}.
\end{itemize}
\end{prp}

\begin{proof}
(a) coincides with (b) by the Hilton--Milnor Theorem, see~\eqref{piwesp}.

\smallskip

We prove that (a) coincides with (c). Since each $x\in GPTW$ is an iterated Whitehead product of $t_1,\dots,t_m$, it is clear that $\Img \varPhi\subset S(\K)$. We now prove the opposite inclusion. 

Take $w\circ \alpha\in S(\K)$, where $w$ is an iterated Whitehead product of $t_1,\dots,t_m$ and $\alpha\in\pi_*(S^{|w|})$. If $|w|=2$, then $w$ is a linear combination of $t_1,\ldots,t_m$. Since $t_i\circ\alpha=0$ for $|\alpha|>2$ and $t_i\circ\id =t_i$, we obtain that $w$ is in the image of $\varPhi$.

So we assume that $|w|>2$ and use Theorem \ref{thm:iterated whitehead product to GPTW} to write $w=\sum_{i=0}^3 p_i\circ \eta^i$, where each $p_i\in\pi_*(\DJ(\K))$ is a Lie polynomial in GPTW elements. By definition, $p_i=g_\K\circ\overline{p}_i$ for some $\overline{p}_i\in\pi_*\bigl(\bigvee_{x\in GPTW}S^{|x|}\bigr)$. By part (3) of Proposition \ref{prp:composition_properties}, it follows that
\[
  w\circ\alpha
  =\bigl(\sum_i g_\K\circ \overline{p}_i\circ\eta^i\bigr)\circ \alpha
  = g_\K\circ \bigl(\sum_i \overline{p}_i\circ\eta^i\bigr)\circ\alpha,
\]  
so that $w\circ\alpha\in\Img \varPhi$, as claimed.
\end{proof}

A simplicial complex $K$ is \emph{flag} if any set of vertices of $\K$ which are pairwise connected by edges spans a simplex. Any flag complex $\K$ is the clique complex of its 1-skeleton (graph) $\K^1$. The \emph{flagification} of $\K$ is the minimal flag complex $\K^{f}$ on the same vertex set $[m]$ that contains $\K$. 

A graph is called \emph{chordal} if each of its cycles with $\ge 4$ vertices has a chord (an edge joining two vertices that are not adjacent in the cycle). Equivalently, a graph is chordal if it has no induced cycles of length more
than three. 

\begin{prp}
\label{prp:chordal+flag=homotopy type of ZK}
Let $\K$ be a flag complex with chordal $1$-skeleton. Then the map
$\widehat g_\K\colon \bigvee_{x\in GPTW}S^{|x|} \to \ZK$ given by the iterated Whitehead products is a homotopy equivalence.
\end{prp}

\begin{proof}
This follows from~\cite[Theorem~3.8]{abramyan_panov}, as any flag complex with chordal $1$-skeleton is totally fillable. We include a proof that is more elementary in this case.

Take a GPTW element $x=c(I,t_j)=[t_{i_1},[t_{i_2},\dots[t_{i_k},t_j]\dots]]\colon S^{k+2}\to\ZK$, where $i_1<\cdots<i_k>j$, $j\notin I=\{i_1,\ldots,i_k\}$ and $j$ is the smallest vertex in a connected component not containing~$i_k$ of the subcomplex~$\mathcal K_{\{i_1,\ldots,i_k,j\}}$. By~\cite[Lemma~3.2]{abramyan_panov}, the Hurewicz image $h(x)\in H_{k+2}(\ZK)$ is represented by the cellular chain
\begin{equation}\label{cellchain}
  S_{i_1}\cdots S_{i_{k-1}}(S_{i_k}D_j+D_{i_k}S_j),
\end{equation}
where $S_i$ is the one-dimensional cell and $D_i$ is the two-dimensional cell in the standard cell decomposition of the $i$th factor $D^2$ in $\ZK=(D^2,S^1)^{\mathcal K}$.

For a flag complex $\mathcal K$ with chordal $1$-skeleton, there is an isomorphism 
\[
  H_{k+2}(\ZK)\cong\bigoplus_{J\subset[m],\,|J|=k+1}\widetilde H^0(\mathcal K_J),
\]
\sloppy
and the cellular chains~\eqref{cellchain} corresponding to different choices of $j\in J=\{i_1,\ldots,i_k,j\}$ generate the part of $H_{k+2}(\ZK)$ corresponding to $\widetilde H^0(\mathcal K_J)$ in the decomposition above. It follows that the map $\widehat g_\K\colon \bigvee_{x\in GPTW}S^{|x|} \to \ZK$
induces an isomorphism in homology groups, so it is a homotopy equivalence by Whitehead's Theorem, as both spaces are simply connected.
\end{proof}

\begin{lmm}
\label{lmm:pt_imply_surjectivity}
Let $\mathcal K^f$ be the flagification of $\K$, and let $\L$ be the simplicial complex consisting of $m$ disjoint points. Then the natural maps $\pi_*(\DJ(\L))\to\pi_*(\DJ(\Kf))$ and $\pi_*(\DJ(\K))\to\pi_*(\DJ(\Kf))$ are split surjections.
\end{lmm}
\begin{proof}
By \cite[Theorem 1.1(b)]{pt}, the natural map $\Omega\DJ(\L)\to\Omega\DJ(\Kf)$ admits a homotopy section. It follows that the map $\pi_*(\DJ(\L))\to\pi_*(\DJ(\Kf))$ admits a section, and hence it is split surjective. Since it factors through $\pi_*(\DJ(\K))$, the map $\pi_*(\DJ(\K))\to\pi_*(\DJ(\Kf))$ is also split surjective.
\end{proof}

Proposition~\ref{prp:chordal+flag=homotopy type of ZK} implies that the map $\varPhi$ in~\eqref{varPhi} is an isomorphism if $\K$ is a flag complex with chordal $1$-skeleton. The main result of this section is the following.

\begin{thm}
\label{thm:Phi inj surj criterion}
Consider the map 
\[
  \varPhi\colon \ZZ\langle t_1,\ldots,t_m\rangle \oplus \pi_*\Bigl(\bigvee_{x\in GPTW}S^{|x|}\Bigr)\to \pi_*(\DJ(\K))
\]
given by iterated Whitehead products corresponding to GPTW elements. Then
\begin{enumerate}
    \item $\varPhi$ is surjective if and only if $\K$ is flag;
    \item $\varPhi$ is injective if and only if $\K^1$ is chordal.
\end{enumerate}
\end{thm}

\begin{proof}
\emph{``If'' part of (1).} Assume that $\K$ is flag. Consider the diagram
\[
  \xymatrix{
  S(\L)
  \ar[r]
  \ar[d]^-{\cong}
  &
  S(\K)
  \ar[d]
  \\
  \pi_*(\DJ(\L))
  \ar@{->>}[r]
  &
  \pi_*(\DJ(\K)).
  }
\] 
It commutes since the definition of the subgroup $S(\K)\subset\pi_{\ast}(DJ(\K))$ is natural for simplicial maps.
The left arrow is an isomorphism by Proposition \ref{prp:chordal+flag=homotopy type of ZK}, and the bottom arrow is surjective by Lemma \ref{lmm:pt_imply_surjectivity}. The commutativity of the diagram implies that the right arrow is surjective. Since $S(\K)=\Img \varPhi$ by Proposition~\ref{SKequiv}, the map $\varPhi$ is surjective.

\smallskip

\emph{``Only if'' part of (1).} Assume that $\K$ is not flag; we are to prove that $\varPhi$ is not surjective. Let $J\subset[m]$ be a missing face with $|J|>2$. Then there is a retraction $r\colon\DJ(\K)\to \DJ(\K_J)$, where $\DJ(\K_J)$ is a fat wedge of $|J|$ copies of~$\mathbb C P^\infty$ and $\pi_{2|J|-1}(\DJ(\K))$ has a direct summand $\pi_{2|J|-1}(\DJ(\K_J))\cong\ZZ$ generated by the higher Whitehead product of $t_j$, $j\in J$. On the other hand, the ordinary Whitehead products of $t_j$ vanish in $\pi_*(\DJ(\K_J))$, so the composite map $r\circ g_\K\colon\bigvee_{x\in GPTW}S^{|x|}\to\DJ(\K_J)$ is null homotopic. It follows that $\varPhi$ does not hit the generator of $\pi_{2|J|-1}(\DJ(\K_J))$, so $\varPhi$ is not surjective.

\smallskip

\emph{``If'' part of (2).} 
Assume that $\K^1$ is a chordal graph. The composite map
\begin{equation}\label{injcomp}
  \ZZ^m\oplus\pi_*\Bigl(\bigvee_{x\in GPTW}S^{|x|}\Bigr)
  \xrightarrow{\varPhi} S(\K)\to S(\Kf)\to \pi_*(\DJ(\Kf))
\end{equation}
is an isomorphism by Proposition~\ref{prp:chordal+flag=homotopy type of ZK} (applied to~$\Kf$). It follows that $\varPhi$ is injective.

\smallskip

\emph{``Only if'' part of (2).}
We show that $\varPhi$ is not injective if $\K^1$ is not chordal. Suppose there exists a chordless cycle $\K_J$ with $|J|\ge4$ vertices. By naturality (and by the fact that $\DJ(\K_J)$ is a retract of $\DJ(\K)$), we can assume that $\K=\K_J$. Then $\K$ is flag, so the map $\varPhi$ is surjective by the first statement. If $\varPhi$ is also injective, then $g\colon\bigvee_{x\in GPTW}S^{|x|}\to\ZK$ induces an isomorphism of homotopy groups, and hence is a homotopy equivalence by Whitehead's Theorem. But by~\cite{MacGavran}, $\ZK$ is homotopy equivalent to a connected sum of sphere products, a contradiction.
\end{proof}

\begin{rmk}
The injectivity of the map $\varPhi$ implies that in the chordal case
the GPTW elements in $\pi_*(\DJ(\K))$ are not subject to any algebraic relations, except for the universal relations that hold in the homotopy groups of any topological space. In particular, Theorem~\ref{thmminrel} does not produce nontrivial relations when $\K^1$ is chordal.
The reason is that the first homology groups $H_1(\K^f_J)$ vanish if the one-skeleton of $\K^f$ is a chordal graph (see e.\,g. \cite[Theorem 1.1]{higher-chordality}).
\end{rmk}

When $\K$ is flag, Theorem~\ref{thm:Phi inj surj criterion} together with Proposition~\ref{SKequiv} imply the following.

\begin{thm}
\label{thm:flagness is surjectivity}
Let $\K$ be a simplicial complex on $[m]$. The following conditions are equivalent:
\begin{enumerate}
\item[(a)] $\K$ is a flag complex;
\item[(b)] the inclusion $(S^2)^{\vee m}\to\DJ(\K)$ induces a surjection on homotopy groups;
\item[(c)] the group $\pi_*(\DJ(\K))$ is generated by elements of the form $x\circ\alpha$, where $x\in\pi_{|x|}(\DJ(\K))$ is an iterated Whitehead product of $t_1,\dots,t_m$, and $\alpha\in\pi_*(S^{|x|})$. 
\end{enumerate} 
In particular, $\K$ is a flag complex if and only if there is an isomorphism $S(\K)\cong\pi_{\ast}(\DJ(\K))$.\qed
\end{thm}

Theorem~\ref{thm:flagness is surjectivity} solves Problem~\ref{prb:main2} in the case of flag complexes. We now consider an approach to the case of a general simplicial complex by comparing it to its flagification.

\begin{prp}
\label{prp:S(K)->S(Kf) is surjective}
The natural map $S(\K)\to S(\Kf)$ is surjective, and is an isomorphism if $\K^1$ is chordal.
\end{prp}

\begin{proof}
Consider the commutative diagram
\[
\xymatrix{
S(\L)
\ar[d]^-{\cong}
\ar[r]
&
S(\K)
\ar[r]
&
S(\Kf)
\ar[d]^-{\cong}
\\
\pi_*(\DJ(\L))
\ar@{->>}[rr]
&&
\pi_*(\DJ(\Kf)).
}
\]
The bottom arrow is surjective by Lemma~\ref{lmm:pt_imply_surjectivity}, and the vertical arrows are isomorphisms by Theorem~\ref{thm:flagness is surjectivity}. Hence $S(\L)\to S(\Kf)$ is surjective, so $S(\K)\to S(\Kf)$ is also surjective.

If $\K^1$ is chordal, then the composite map~\eqref{injcomp} is an isomorphism, in which the left map is injective and the right map is an isomorphism by~Theorem~\ref{thm:Phi inj surj criterion}. It follows that $S(\K)\to S(\Kf)$ is also an isomorphism.
\end{proof}

\begin{question}
\label{cnj:PTinverse-weaker}
For a simplicial complex $\K$, consider $g_\K\colon\bigvee_{x\in GPTW}S^{|x|}\to\DJ(\K)$ and the homomorphism $\sigma\colon\pi_*(\DJ(\Kf))\to\pi_*(\DJ(\K))$ induced by the section constructed in \cite{pt} (see Lemma \ref{lmm:pt_imply_surjectivity}). Is it true that $\Img (g_\K)_*\subset\Img\sigma$?
\end{question}

If the answer to the question above is positive, then $S(\K)\to S(\Kf)$ is an isomorphism, and
$  \Img\bigl(\pi_*((S^2)^{\vee m})\to\pi_*(\DJ(\K))\bigr)
  \cong \pi_*(\DJ(\Kf))$.

\subsection*{Geometric interpretation}
The Whitehead products generating the subgroup $S(\K)$ in the homotopy groups of $\DJ(\K)$ are interpreted geometrically as a map $g_\K$ from a wedge of spheres corresponding to GPTW elements. Other related maps are shown in the following diagram: 
\begin{equation}
\label{s2djfib}
\xymatrix{
\bigvee\limits_{\alpha\in\Zm}(S^{|\alpha|+1})^{\vee \b_0(\K_\alpha)}
\ar[rrd]^-{f_\K}
\ar[rd]^-{\widehat f_\K}
\ar[d]^-{r'_{\K}}\\
\bigvee\limits_{J\subset[m]}(S^{|J|+1})^{\vee \b_0(\K_J)}
\ar[rd]^-{\widehat g_\K}
\ar@<1ex>@/^/[u]^-{q'_{\K}}
&
(C\Omega S^2,\Omega S^2)^\K
\ar[d]^-{r_{\K}}
\ar[r]
&
(S^2,\ast)^\K
\ar[r]
\ar[d]^-{t_\K}
&
(S^2)^{\times m}
\ar[d]\\
&
(CS^1,S^1)^\K
\ar@<1ex>@/^/[u]^-{q_{\K}}
\ar[r]^-{}
&
(\CC P^\infty,\ast)^\K
\ar[r]
&
(\CC P^\infty)^{\times m}
}
\end{equation} 
Here, the horizontal lines are homotopy fibrations of polyhedral products (see Proposition~\ref{loopsec}).
The map $t_\K$ is induced by the inclusion $S^2\to\CC P^\infty$. The map $r_{\K}$ has a right homotopy inverse $q_{\K}$ and both are induced by the natural retraction $S^1\to\Omega S^2\to S^1$. 
The map $\widehat g_\K$ is the lift to $\ZK=(CS^1,S^1)^\K$ of the map $g_\K$ given by the wedge of GPTW elements. In the notation~\eqref{GPTWTheta} the latter is
\[
  g_\K=\bigvee_{J\subset[m]}
  \bigvee_{j\in\Theta_\K(J)}c(J\setminus j,j;\underline t)\colon \bigvee_{J\subset[m]}\bigvee_{j\in\Theta_\K(J)}S^{|J|+1}\to
  (\CC P^\infty,\ast)^\K
\]
where each GPTW element $c(J\setminus j,j;\underline t)$ is an iterated Whitehead product of the maps $t_i$ with no repeating indices. The map $f_\K$ is  defined as a similar wedge of iterated Whitehead products $c(\alpha-j,j;\underline{s})\colon S^{|\alpha|+1}\to (S^2,\ast)^\K$, this time with repeating indices:
\[
  f_\K=\bigvee_{\alpha\in\Zm}
  \bigvee_{j\in\Theta_\K(\alpha)}c(\alpha-j,j;\underline{s})\colon 
  \bigvee_{\alpha\in\Zm}
  \bigvee_{j\in\Theta_\K(\alpha)} S^{|\alpha|+1}\to (S^2,\ast)^\K,
\]
and the map $\widehat f_\K$ is the lift of $f_\K$ to $(C\Omega S^2,\Omega S^2)^\K$.
The map $q'_\K$ is the inclusion of wedge summands corresponding to squarefree multi-indices. Finally, $r'_\K$ is the map expressing each iterated Whitehead product $c(\alpha-j,j;\underline{s})$ via GPTW elements. Namely, we express the iterated Whitehead product 
\begin{equation}\label{cts}
  c(\alpha-j,j;\underline{t})=t_\K\circ c(\alpha-j,j;\underline{s})\colon S^{|\alpha|+1}\to(\CC P^\infty,\ast)^\K 
\end{equation}
as a sum of polynomials in GPTW elements $c(I\setminus i,i;\underline t)$ composed with $\eta,\eta^2,\eta^3$, using Theorem~\ref{thm:iterated whitehead product to GPTW} (an explicit expression can be obtained using Lemma~\ref{lmm:computing c(alpha-j,j)}  and \cite[Algorithm 5.4]{vylegzhanin}). This expression for $c(\alpha-j,j;\underline{t})$ defines a map $S^{|\alpha|+1}\to \bigvee\limits_{J\subset[m]}(S^{|J|+1})^{\vee \b_0(\K_J)}$, and $r'_\K$ is the wedge of these maps over $\alpha\in\ZZ^m_{\ge0}$ and $j\in\Theta_\K(\alpha)$. It is clear from this definition that $r'_\K\circ q'_\K=\id$, and $r'_\K$ is nontrivial only on a finite number of wedge summands. Furthermore, \eqref{cts} together with the functoriality of Whitehead products implies that $g_\K\circ r'_\K=t_\K\circ f_\K$. The lifts $\widehat f_\K$ and $\widehat g_\K$ are unique up to homotopy, which implies that $\widehat g_\K\circ r'_\K\simeq r_\K\circ \widehat f_\K$. By~\cite[Lemma 5.4]{vv} or Proposition~\ref{prp:f and s commute with lifts} below, $\widehat{f}_\K\circ q'_\K\simeq q_\K\circ \widehat{g}_\K$. Thus  diagram~\eqref{s2djfib} is homotopy commutative.

The map $\widehat g_\K$ in~\eqref{s2djfib} is surjective in homotopy groups when $\K$ is flag, and it is a homotopy equivalence when $\K$ is flag with chordal $\K^1$ (Proposition~\ref{prp:chordal+flag=homotopy type of ZK}). Furthermore, for the disjoint $m$ point complex $\L$, the map $\widehat f_\L$ is also a homotopy equivalence by \cite[Theorem~7.4]{theriault_dual}. (Conjecturally, the map $\widehat f_\K$ is a homotopy equivalence when $\K$ is flag with chordal $\K^1$.) Here is the explicit description of the maps in diagram~\eqref{s2djfib} for the case $\K=\L$ that uses the identity from Lemma~\ref{lmm:computing c(alpha-j,j)}.

\begin{prp}\label{diagDJL}
There is a commutative diagram of homotopy fibrations    
\[
\xymatrix{
\bigvee_{\alpha\in\Zm}\bigvee_{j\in \supp\alpha\setminus\max(\supp\alpha)}
S^{|\alpha|+1}
\ar[rrr]^-{\vee c(\alpha-j,j;\underline s)}
\ar[d]^-{r'_{\L}}
&&&
(S^2)^{\vee m}
\ar[r]
\ar[d]
&
(S^2)^m\ar[d]\\
\bigvee_{J\subset[m]}\bigvee_{j\in J\setminus\max(J)}
S^{|J|+1}
\ar[rrr]^-{\vee c(J\setminus j,j;\underline t)}
\ar@<1ex>@/^/[u]^-{q'_{\L}}
&&&
(\CC P^\infty)^{\vee m}
\ar[r]
&
(\CC P^\infty)^m
}
\]
Here, $q'_\L$ is the inclusion of wedge summands and $r'_\L$ is the wedge of maps
\[
  s_{J,j}\circ \eta^{|\alpha|-|J|}+
  \sum_{\begin{smallmatrix}
    J\setminus j=A\sqcup B:\\
    \max(A)>j,\\
    \max(J)\in B
  \end{smallmatrix}}\pm [s_{A\sqcup j,j},s_{B\sqcup j,j}]\circ\eta^{|\alpha|-|J|-1}\colon
  S^{|\alpha|+1}\to
  \bigvee_{J,j}S^{|J|+1},
\]  
where $J=\supp\alpha$ and $s_{J,j}\colon S^{|J|+1}\to \bigvee_{J,j}S^{|J|+1}$ is the inclusion of a single sphere in the wedge, and the second summand above is zero if $\alpha_j=1$.
In particular, $r'_\L\colon \bigvee_{\alpha,j}S^{|\alpha|+1}\to\bigvee_{J,j}S^{|J|+1}$ vanishes on spheres with $|\alpha|-|\supp\alpha|\geq 4$, i.\,e. vanishes on all but finitely many spheres in the wedge.\qed
\end{prp}

\begin{exm}\label{lm2}
Let $\K=\L$ with $m=2$, so that $\alpha=(k,\ell)$. We obtain the diagram
\[
\xymatrix{
\bigvee\limits_{k,l\geq 1}S^{k+\ell+1}
\ar[d]^{r'}
\ar[r]^-{f}
&
S^2\vee S^2
\ar[r]
\ar[d]
&
S^2\times S^2
\ar[d]\\
S^3
\ar@/^/[u]^-{q'}
\ar[r]^-{[t_2,t_1]}
&
\CC P^\infty\vee\CC P^\infty
\ar[r]
&
\CC P^\infty\times\CC P^\infty.
}
\]
where $f$ is the wedge of the maps
\[
  c((k-1,\ell),1;\underline s)=[\underbrace{s_1,[\dots,[s_1}_{k-1\text{ times}},[\underbrace{s_2,[\dots,[s_2}_{\ell\text{ times}},s_1]\dots]]]\dots]]\colon S^{k+\ell+1}\to S^2\vee S^2,
\]
the map $q'$ is the inclusion of a wedge summand, and $r'$ is the wedge of the identity map when $k+\ell=2$ and iterated Hopf maps $\eta^{k+\ell-2}=\eta\circ\dots\circ \eta\colon S^{k+\ell+1}\to S^3$ when $k+\ell>2$. Notice that $\eta^{k+\ell-2}$ is trivial for $k+\ell\geq 6$. 
\end{exm}

\begin{exm}
Let $\K=\L$ on $m=3$ vertices. Then 
\[
  \Z_\L\simeq \bigvee_{J\subset[3]}\bigvee_{j\in J\setminus\max(J)}
  S^{|J|+1}= (S^3)^{\vee 3}\vee (S^4)^{\vee 2}
\]
is mapped to $(\CC P^\infty)^{\vee 3}$ as the wedge of GPTW elements
\[
  [t_2,t_1],\; [t_3,t_1],\; [t_3,t_2],\; [t_2,[t_3,t_1]]\text{ and }
  [t_1,[t_3,t_2]]. 
\]  
Here $[t_3,t_1]=c(\{13\}\!\setminus\!1,1;\underline t)$ and $[t_1,[t_3,t_2]]=c(\{123\}\!\setminus\!2,2;\underline t)$, for example. 

We index the spheres in the wedge by their corresponding brackets and consider the map 
\[
  r'\colon \bigvee_{\alpha\in\ZZ^3}\bigvee_{j\in \supp\alpha\setminus\max(\supp\alpha)}S^{|\alpha|+1}
  \to S^3_{[t_2,t_1]}\vee S^3_{[t_3,t_1]}\vee S^3_{[t_3,t_1]}
  \vee S^4_{[t_2,[t_3,t_1]]} \vee S^4_{[t_1,[t_3,t_2]]}.
\]  
The inclusion $s_{J,j}\colon S^{|J|+1}\to \bigvee_{J,j}S^{|J|+1}$ 
takes $S^3$ or $S^4$ to the corresponding wedge summand; for instance, $s_{\{123\},1}$ takes $S^4$ to $S^4_{[t_2,[t_3,t_1]]}$.

When $|\supp\alpha|=2$, the map $r'$ takes the wedge summand $S^{|\alpha|+1}$ to the corresponding wedge summand $S^3$ as described in Example~\ref{lm2}. Namely, $r'|_{S^{|\alpha|+1}}=S_{J,j}\circ\eta^{|\alpha|-J}$ according to the formula in Proposition~\ref{diagDJL}. For example, for $\alpha=(\alpha_1,0,\alpha_3)$, the map $r'$ takes $S^{|\alpha|+1}$ to $S^3_{[t_3,t_1]}$ by the power $\eta^{\alpha_1+\alpha_3-2}$ of the Hopf map (or identity when $\alpha_1=\alpha_3=1$).

Now consider $\alpha=(\alpha_1,\alpha_2,\alpha_3)$ with $\alpha_1\ge2$, $\alpha_2\ge1$, $\alpha_3\ge1$, so that $J=\supp\alpha=\{123\}$ and take $j=1$. Then the formula in Proposition~\ref{diagDJL} gives
\[
  r'|_{S^{|\alpha|+1}}=s_{\{123\},1}\circ\eta^{|\alpha|-3}\pm
  [s_{\{12\},1},s_{\{13\},1}]\circ\eta^{|\alpha|-4}.
\]
For instance, when $\alpha=(2,1,1)$, the map $r'|_{S^5}\colon S^5\to S^3\vee S^3\vee S^4$ is given by $\pm[\iota_3,\iota_3]+\iota_4\circ\eta$, and its composite with the map to $(\CC P^\infty)^{\vee 3}$ is given by
\[
  [[t_2,t_1],[t_3,t_1]]+[t_2,[t_3,t_1]]\circ\eta.
\]
\end{exm}

\section{Quasi-Lie algebras and Problem~\ref{prb:main1}}
\label{sec:ql(K)}

In this section we address Problem~\ref{prb:main1} by studying the quasi-Lie subring
\[
  QL(\K)\subset \pi_*(\DJ(\K))
\]  
generated by the standard elements $t_1,\dots,t_m\in\pi_2(\DJ(\K))$.

Consider the graded abelian group
\[
  A_n=\ZZ\bigl\langle\iota,\eta^k,[\iota,\iota]\circ\eta^k,[\iota,[\iota,\iota]]\circ\eta^k,\;0\leq k\leq 3\bigr\rangle\subset\pi_*(S^n)
\]
described in Appendix~\ref{sec:groups A_n}.  Let 
\[
  Q(\K)=\ZZ\langle t_1,\ldots,t_m\rangle
  \oplus\bigoplus_{b\in\BC(GPTW)}A_{|b|}
\]  
be the subgroup of the group
\[
  \ZZ\langle t_1,\ldots,t_m\rangle
  \oplus\bigoplus_{b\in\BC(GPTW)}\pi_*(S^{|b|})\cong \ZZ^m\oplus\pi_*\Big(\bigvee_{x\in GPTW}S^{|x|}\Big),
\]
and let 
\[
  \phi\colon Q(\K)\to\pi_*(\DJ(\K))
\]
be the restriction of the map $\varPhi\colon \ZZ^m\oplus\pi_*(\bigvee_{x\in GPTW} S^{|x|})\to\pi_*(\DJ(\K))$.

\begin{thm}
\label{thm:QL(K) description}
Let $\K$ be a simplicial complex and $\K^f$ its flagification. Then: 
\begin{enumerate}
  \item $QL(\K)=\Img\phi$;
  \item the natural homomorphism of quasi-Lie algebras $QL(\K)\to QL(\K^f)$ is surjective;
  \item if $\K^1$ is a chordal graph, then $Q(\K)\cong QL(\K)\cong QL(\K^f)$.
\end{enumerate}
\end{thm}

\begin{proof}
We first prove the inclusion $\Img\phi\subset QL(\K)$.
For each $b\in\BC(GPTW)$, the map $\phi$ takes the identity $\iota\in A_{|b|}\subset\pi_*(S^{|b|})$ to the basic commutator $b$ of GPTW elements in $\pi_*(\DJ(\K))$. Similarly, 
\begin{equation}\label{phib}
  \phi(\iota\circ\eta^k)=b\circ\eta^k,\quad \phi([\iota,\iota]\circ\eta^k)=[b,b]\circ\eta^k,\quad
  \phi([\iota,[\iota,\iota]]\circ\eta^k)=[b,[b,b]]\circ\eta^k.
\end{equation}
By definition, $t_1,\dots,t_m\in QL(\K)$, and each GPTW element is a nested commutator in $t_1,\dots,t_m$. It follows that $\phi(\iota)=b\in QL(\K)$. 

Now we prove that $\phi(\iota\circ\eta^k)=b\circ\eta^k\in QL(\K)$ for all $b\in\BC(GPTW)$ and $k\geq 0$.
Note that $|b|\geq 3$, so $\eta$ must be a suspension. We proceed by induction on the number of nested commutators in $b$. If $b\in GPTW$, then we can write $b=[t_{j_1},c]$. By Lemma \ref{lmm:adding_letter}, $b\circ\eta^k=[t_{j_1},[t_{j_1},\dots[t_{j_1},c]\dots]]$, so $b\circ\eta^k\in QL(\K)$. Now suppose $b=[b',b'']$ for $b',b''\in\BC(GPTW)$. Then
$b\circ\eta^k=[b',b'']\circ\eta^k=[b'\circ\eta^k,b'']$ by Proposition~\ref{prp:properties of hopf}~(1).
Here $b'\circ\eta^k\in QL(\K)$ by induction, and $b''\in QL(\K)$. It follows that $b\circ\eta^k\in QL(\K)$.

It remains to prove that $[b,b]\circ\eta^k$ and $[b,[b,b]]\circ\eta^k$ belong to $QL(\K)$. We have
$
  [b,b]\circ\eta^k=  [b\circ\eta^k,b]\in QL(\K),
$
since $\eta$ is a suspension and $b,b\circ\eta^k\in QL(\K)$. Hence, $[b\circ\eta^k,b]\in QL(\K)$. Similarly, $[b,[b,b]]\circ\eta^k=[b\circ\eta^k,[b,b]]\in QL(\K)$.

\smallskip

Now we prove the inclusion $QL(\K)\subset\Img\phi$.
By Theorem \ref{thm:iterated whitehead product to GPTW}, $QL(\K)$ is additively generated by $t_1,\dots,t_m$ and elements of the form $y\circ\eta^k$, where $y$ is an iterated Whitehead product on $GPTW$, and $0\leq k\leq 3$. By  Corollary~\ref{crl:lie_polynomial_normal_form}, any such $y$ is a linear combination of elements $b$, $[b,b]$ and $[b,[b,b]]$, where $b\in\BC(GPTW)$. It follows that $QL(\K)$ is additively generated by elements $t_1,\dots,t_m$, $b\circ\eta^k$, $[b,b]\circ\eta^k$ and $[b,[b,b]]\circ\eta^k$, which belong to $\Img\phi$ by~\eqref{phib}. This proves statement~(1).

\smallskip

Now we prove (2). Since $Q(\K)$ is defined in terms of $\K^1$, the natural map $Q(\K)\to Q(\Kf)$ is an isomorphism. The surjectivity of $QL(\K)\to QL(\K^f)$ follows by considering the commutative diagram
\[
  \xymatrix{
  Q(\K) \ar[r]^\phi\ar[d]^\cong & QL(\K) \ar[d] \\
  Q(\K^f) \ar[r]^\phi & QL(\K^f)
  }
\]
in which the horizontal arrows are surjective by statement~(1).

\smallskip

Finally, suppose that $\K^1$ is a chordal graph. Then $\varPhi\colon \ZZ^m\oplus\pi_*(\bigvee_{x\in GPTW} S^{|x|})\to\pi_*(\DJ(\K^f))$  is injective by Theorem~\ref{thm:Phi inj surj criterion}~(2), so its restriction $\phi\colon Q(\Kf)\to QL(\Kf)$ is also injective. Since $\phi\colon Q(\Kf)\to QL(\Kf)$ is both surjective and injective, it is an isomorphism. Now statement~(3) follows from the commutative diagram above.
\end{proof}

\begin{crl}
If $\K^1$ is chordal, then the graded abelian group $QL(\K)$ has no $p$-torsion for $p>3$.
\end{crl}
\begin{proof}
By Theorem \ref{thm:QL(K) description}, $QL(\K)\cong Q(\K)$ is a direct sum of $\ZZ^m$ and groups $A_n$. The latter groups are described in Appendix~\ref{sec:groups A_n} and have no $p$-torsion for $p>3$.
\end{proof}

\begin{exm}
Let $\K$ be a chordal flag complex, and let
$[t_i,t_j],[t_k,t_\ell]\in GPTW$, where $(i,j)>(k,\ell)$ lexicographically. Then $b=[[t_i,t_j],[t_k,t_\ell]]\in\BC(GPTW)$ is a basic commutator of degree $5$. For each $\alpha\in\pi_n(S^5)$, there is the corresponding element $b\circ\alpha\in\pi_n(\DJ(\K))$. By Theorem~\ref{thm:QL(K) description}, $b\circ\alpha\in QL(\K)$ if and only if $\alpha\in A_5$.
\end{exm}

\begin{question}
Is the surjective map $QL(\K)\to QL(\Kf)$ an isomorphism? (This is true if $\K^1$ is chordal.)
\end{question}

If $\K^1$ is a chordal graph, Theorem \ref{thm:QL(K) description} provides an additive description of $QL(\K)=QL(\K^f)$, and the proof gives an algorithm that computes the bracket of any two elements.
However, we do not have an explicit description of the multiplicative structure of $QL(\K)$, e.g., by generators and relations, except for the simplest case considered in the next example.

\begin{exm}[$\K=\bullet\,\bullet$] 
Let $\K=\{\varnothing,\{1\},\{2\}\}$ be a two point complex. Then 
$\BC(GPTW)=GPTW=\{[t_2,t_1]\}$, and hence
\[ 
  QL(\K)=\ZZ\langle t_1, t_2\rangle\oplus[t_1,t_2]\circ A_3.
\]
We have $A_3=\ZZ\langle\iota\rangle\oplus\ZZ_2\langle\eta,
\eta^2,\eta^3\rangle$. Hence, $QL(\K)\cong\ZZ^2\oplus\ZZ\oplus\ZZ_2\oplus\ZZ_2\oplus\ZZ_2$ is a finitely generated graded abelian group with the following Lie relations, which follow from Proposition~\ref{qla2p}:
\begin{gather*}
  [t_1,t_1]=[t_2,t_2]=0,\quad 
  [[t_1,t_2]\circ\eta^k,[t_1,t_2]\circ\eta^\ell]=0,\\
  [t_1,[t_1,t_2]\circ\eta^k]=[t_2,[t_1,t_2]\circ\eta^k]=[t_1,t_2]\circ\eta^{k+1},\quad 0\le k\le 3,\;\eta^4=0.
\end{gather*}
Note that in this case $[a,a]=0$ for all $a\in QL(\K),$ so $QL(\K)$ is actually a graded Lie ring, and $QL(\K)\otimes_\ZZ\ZZ_2$ is a graded Lie algebra over $\ZZ_2$.
Using a result of Veryovkin~\cite{veryovkin}, it can be described by generators and relations as follows:
\[QL(\K)\otimes_\ZZ\ZZ_2=\FL_{\ZZ_2}(t_1,t_2)/\mathcal{R}\] 
where $\mathcal{R}$ is the Lie ideal generated by the relations 
\begin{gather*} 
[t_1,[t_1,t_2]]=[t_2,[t_1,t_2]],\quad
[[t_1,t_2],[t_1,[t_1,t_2]]]=0,\\
[t_1,[t_1,[t_1,[t_1,[t_1,t_2]]]]]=0.
\end{gather*}
\end{exm}

\begin{exm}[$\K=\bullet\bullet\!\!\!-\!\!\bullet$]
Let $\K=\{\varnothing,\{1\},\{2\},\{3\},\{2,3\}\}$. The homotopy fibration~\eqref{eq:zk-dj fibration} becomes
\[
  S^3\vee S^3\vee S^4\to 
  \CC P^\infty\vee(\CC P^\infty\times\CC P^\infty)\to
  \CC P^\infty\times\CC P^\infty\times\CC P^\infty
\]  
and $\ZK\to\DJ(\K)$ is the inclusion of a wedge of three spheres corresponding to the GPTW elements
\[GPTW=\{a=[t_2,t_1],\;b=[t_3,t_1],\;c=[t_2,[t_3,t_1]]\}.\] 
Hence, $\BC(GPTW)=BC(a,b,c)$ is an infinite set. The first few graded components of $\BC(GPTW)$ are:
\begin{gather*}
\BC(GPTW)_3=\{a,b\},\quad \BC(GPTW)_4=\{c\},\quad \BC(GPTW)_5=\{[b,a]\},\\
\BC(GPTW)_6=\{[c,a],[c,b]\},\quad \BC(GPTW)_7=\{[a,[b,a]],[b,[b,a]]\},\\
\BC(GPTW)_8=\{[a,[c,a]],[b,[c,a]],[[c,b],b]\},\\
\BC(GPTW)_9=\{c,[[c,a]],[c,[c,b]],[a,[a,[b,a]]],[a,[b,[b,a]]],[b,[b,[b,a]]]\}.
\end{gather*}
In topological degree $10$, non-nested basic commutators $[[c,a],[b,a]]$ and $[[c,b],[b,a]]$ appear.
It follows that additively 
\begin{align*}
\pi_*(\DJ(\K))&=\ZZ\langle t_1,t_2,t_3\rangle
\oplus\{a,b\}\circ \pi_*(S^3)
\oplus\{c\}\circ\pi_*(S^4)
\oplus\{[b,a]\}\circ\pi_*(S^5)
\oplus\cdots\\
QL(\K)&=\ZZ\langle t_1,t_2,t_3\rangle
\oplus\{a,b\}\circ A_3
\hspace{6.1mm}\oplus\{c\}\circ\ A_4
\hspace{5.2mm}\oplus\{[b,a]\}\circ A_5
\hspace{5.8mm}\oplus\cdots
\end{align*}
For example, $\pi_5(S^3)=\ZZ_2\langle\eta^2\rangle\subset A_3$, $\pi_5(S^4)=\ZZ_2\langle\eta\rangle\subset A_4$, $\pi_5(S^5)=\ZZ\langle\iota\rangle$, so
\[
  QL(\K)_5= \ZZ_2\langle a\circ\eta^2, b\circ\eta^2\rangle \oplus \ZZ_2\langle c\circ\eta\rangle \oplus \ZZ\langle[b,a]\rangle   
  \cong\ZZ_2^{\oplus 3}\oplus\ZZ=\pi_5(\DJ(\K)).
\]
  
As explained in \S \ref{subsec:mac-wedge-of-spheres}, the most nontrivial part in the calculation of Whitehead products is the computation of $[t_i,x],$ where $x\in\BC(GPTW).$ We first compute the action on the GPTW generators: 
\begin{align*} 
 [t_1,a] &=a\circ\eta,~[t_2,a]=a\circ\eta,~[t_3,a]=[t_3,[t_2,t_1]]=-c; \\
 [t_1,b] &=b\circ\eta,~[t_2,b]=c,~[t_3,b]=b\circ\eta; \\
 [t_1,c] &=[t_1,[t_2,[t_3,t_1]]]=-[[t_1,t_2],[t_3,t_1]]-[t_2,[t_1,[t_3,t_1]]] \\ 
 & \hspace{2.68cm}=-[a,b]-[t_2,[t_3,t_1]\circ\eta]=
 [b,a] +c\circ\eta; \\ 
 [t_2,c] &=c\circ\eta; \\ 
 [t_3,c] &=[t_3,[t_2,[t_3,t_1]]]=-[[t_3,t_2],[t_3,t_1]]-[t_2,[t_3,[t_3,t_1]]] \\ 
& \hspace{2.67cm}=0-[t_2,[t_3,t_1]\circ\eta]=c\circ\eta, 
\end{align*} 
where $[t_2,t_3]=0$ since vertices $2$ and $3$ are joined by an edge.
This allows for a recursive computation of $[t_i,x]$, $x\in QL(\K)$. For example,
\[
  [t_1,[c,a]]=-[[t_1,c],a]-[c,[t_1,a]]=-[[b,a],a] +[c,a]\circ\eta+[c,a]\circ\eta=[a,[b,a]].
\]  
\end{exm}

We now collect some general phenomena seen in this case.

\begin{prp}
Let $\K=\bullet\bullet\!\!\!-\!\!\bullet$. Then
\begin{enumerate}
    \item The subgroup $QL(\K)\subset\pi_*(\DJ(\K))$ is not a direct summand.
    \item $QL(\K)$ contains $3$-torsion.
    \item $[c,c]\neq 0$ and $[[b,a],[b,a]]\neq 0$ in $QL(\K)\otimes_\ZZ \ZZ_2$. In particular, $QL(\K)\otimes_\ZZ \ZZ_2$ is not a graded Lie algebra over~$\ZZ_2$.
\end{enumerate}
\end{prp}
\begin{proof}
(1) Consider degree $6$. The group $\pi_6(S^3)$ has order $12$, with a generator that we temporarily denote by~$\omega$ (see Proposition~\ref{pin+k}). We have $\eta^3=6\omega$, since $\eta^3\in\pi_6(S^3)$ has order~$2$. It follows that
the group $QL(\K)_6$ has a direct summand $\ZZ_2\langle a\circ\eta^3,b\circ\eta^3\rangle$ included in the corresponding direct summand $\ZZ_{12}\langle a\circ\omega, b\circ\omega\rangle$ of $\pi_6(\DJ(\K))$. 

(2)
A nontrivial $3$-torsion component in $QL(\K)$ appears as the direct summand $\ZZ_3\langle[c,[c,c]]\rangle\subset\pi_{10}(\DJ(\K))$, since $[c,[c,c]]=c\circ [\iota_{4},[\iota_{4},\iota_{4}]]$ and $[\iota_4,[\iota_4,\iota_4]]\in\pi_{10}(S^4)$ is a nontrivial element of order $3$. 

(3) By \cite[(5.8)]{toda}, the relation $[\iota_4,\iota_4]=\pm(2\nu-E\nu')$ holds in $\pi_7(S^4)$. 
It follows that $[\iota_4,\iota_4]$ is nontrivial and not divisible by two. As there is an inclusion of a direct summand $\{c\}\circ\pi_{\ast}(S^{4})\subset\pi_{\ast}(DJ(\K))$, it follows that the element $[c,c]=c\circ[\iota_4,\iota_4]\in QL(\K)$ is nontrivial 
in $\pi_7(\DJ(\K))\otimes_\ZZ\ZZ_2$.

Similarly, by \cite[(5.10)]{toda}, the group $\pi_9(S^5)\cong\ZZ_2$ is generated by $[\iota_5,\iota_5]$. As there is an inclusion of a direct summand $\{[b,a]\}\circ\pi_{\ast}(S^{5})\subset\pi_{\ast}(DJ(\K))$, it follows that the element $[[b,a],[b,a]]=[b,a]\circ [\iota_{5},\iota_{5}]\in QL(\K)_9$ is nontrivial.
\end{proof}

\section{Generalisation to polyhedral products}
\label{sec:spheres and dj} 

This section describes how the results on iterated Whitehead products in the homotopy groups of $\DJ(\K)=(\mathbb C P^\infty,\ast)^\K$ can be generalised to polyhedral products of the form $(\X,\ast)^{\K}$, and in a refined way to polyhedral products of the form $(\Sigma\Y,\ast)^{\K}$. 

\begin{lmm}\label{exfib2}
Let $F\to E\overset{p}\longrightarrow B$ be a fibration such that $\Omega p$ has a right homotopy inverse. Then, for any based space $A$, there is a short exact sequence of groups
$$1\to [\Sigma A,F]\to [\Sigma A,E]\to [\Sigma A,B]\to 1.$$
\end{lmm}

\begin{proof}
Consider the exact sequence of pointed sets 
\[
  [\Omega B,\Omega E]\xrightarrow{(\Omega p)_*} [\Omega B,\Omega B]\xrightarrow{\delta_*} [\Omega B,F],
\]  
where $\delta\colon\Omega B\to F$ is the connecting map of the fibration. Since $\id_{\Omega B}$ is in the image of $(\Omega p)_*$ by assumption, the map $\delta$ is null homotopic. It follows that the left and right maps in the exact sequence
\[
  [A,\Omega^2 B]\overset{(\Omega\delta)_*}\longrightarrow [A,\Omega F]\to [A,\Omega E]\to [A,\Omega B]\overset{\delta_*}\longrightarrow [A,F]
\]  
are trivial, resulting in a short exact sequence of groups $1\to [A,\Omega F]\to [A,\Omega E]\to [A,\Omega B]\to 1$. Taking adjoints, we obtain the required short exact sequence.
\end{proof}

We use notation~\eqref{cprod} for iterated Whitehead products in a polyhedral product.

\begin{prp}
Let $\alpha\in\Zm$, $|\alpha|\geq 2$, $j\in \supp\alpha$. 

The iterated Whitehead product $c(\alpha-j,j,t_\X)\colon\Sigma(\Omega \X)^{\wedge\alpha}\to (\X,\ast)^\K$ lifts uniquely through the homotopy fibration
\[
  (C\Omega\X,\Omega\X)^\K\to(\X,\ast)^\K\to \prod_{i=1}^m X_i
\]
to a map
\[
  \widehat{c}(\alpha-j,j;t_\X)\colon \Sigma(\Omega \X)^{\wedge\alpha}\to 
  (C\Omega\X,\Omega\X)^\K.
\]

Similarly, the iterated Whitehead product $c(\alpha-j,j;\incl_{\Sigma\Y})\colon \Sigma\Y^{\wedge\alpha}
\to(\Sigma\Y,\ast)^\K$
lifts uniquely through the homotopy fibration
\[
  (C\Omega\Sigma\Y,\Omega\Sigma\Y)^\K\to(\Sigma\Y,\ast)^\K\to \prod_{i=1}^m \Sigma Y_i
\] to a map
\[
  \widehat{c}(\alpha-j,j;\incl_{\Sigma\Y})\colon
  \Sigma \Y^{\wedge\alpha}\to 
  (C\Omega\Sigma\Y,\Omega\Sigma\Y)^\K.
\]
\end{prp}

\begin{proof}
The definition of the Whitehead product implies that the composition of $c(\alpha-j,j,t_\X)\colon\allowbreak\Sigma(\Omega \X)^{\wedge\alpha}\to (\X,\ast)^\K$ with the map $(\X,\ast)^\K\to \prod_{i=1}^m X_i$ is null homotopic, and therefore $c(\alpha-j,j,t_\X)$ lifts to a map to $(C\Omega\X,\Omega\X)^\K$. The uniqueness of the lift follows from Lemma~\ref{exfib2}, since the fibration in question has a section after looping by Proposition~\ref{loopsec}. The second statement is proved similarly.
\end{proof}

We now define the map
\[
  g_\K=\bigvee c(J\setminus j,j;t_\X)\colon
  \bigvee_{J\subset[m]}\bigvee_{j\in\Theta_\K(J)}
  \Sigma(\Omega\X)^{\wedge J} \to (\X,\ast)^\K
\]
given by the wedge of GPTW elements (see \S\ref{notiwp}), and its lift
\[
  \widehat g_\K=\bigvee \widehat{c}(J\setminus j,j;t_\X)\colon
  \bigvee_{J\subset[m]}\bigvee_{j\in\Theta_\K(J)}
  \Sigma(\Omega\X)^{\wedge J} \to (C\Omega\X,\Omega\X)^\K.
\]  
These maps appear in a natural generalisation of diagram~\eqref{s2djfib}.
We have the unit and counit maps $\ev_X\colon\Sigma\Omega X\to X$ and $E_Y\colon Y\to\Omega\Sigma Y$ for the pair of adjoint functors $\Sigma$ and $\Omega$. 
The maps $r_\K:=(C\Omega \ev_\X,\Omega\ev_\X)^\K$ and $t_\K:=(\ev_{\X},\ast)^{\K}$ define a map between the homotopy fibrations in the rows of the following diagram:
\begin{equation}
\label{phpfib}
\xymatrix@R=0.8pc{
  & (C\Omega\Sigma\Omega\X,\Omega\Sigma\Omega\X)^\K
  \ar[dd]^-{r_{\K}}  \ar[r]^-{\iota}
  & (\Sigma\Omega\X,\ast)^\K
  \ar[dd]^-{t_\K} \ar[r]
  & \prod_{i=1}^m\Sigma\Omega X_i
  \ar[dd]^-{\prod\ev_{X_i}} \\
  \Sigma(\Omega\X)^{\wedge\alpha}
  \ar@/^0.8pc/[ur]^-{\widehat{c}(\alpha-j,j;\incl_{\Sigma\Omega\X})
  \ \ \ }
  \ar@/_0.8pc/[dr]_-{\widehat{c}(\alpha-j,j;t_\X)\ }\\
  &(C\Omega\X,\Omega\X)^\K
  \ar@<1ex>@/^/[uu]^-{q_{\K}}
  \ar[r]^-{}
  & (\X,\ast)^\K
  \ar[r]
  & \prod_{i=1}^m X_i.
}
\end{equation} 
The map $r_\K$ has a right homotopy inverse $q_\K:=(C E_{\Omega\X},E_{\Omega\X})^\K$, since the composite $\Omega\ev_X\circ E_{\Omega X}\colon\Omega X\to\Omega\Sigma\Omega X\to\Omega X$ is the identity map. 

The next proposition, generalising~\cite[Lemma 5.4]{vv}, shows that the left triangles in~\eqref{phpfib} are also commutative.

\begin{prp}
\label{prp:f and s commute with lifts}
For $\alpha\in\Zm$, $|\alpha|\geq 2$ and $j\in\supp\alpha$, the following identities hold up to a homotopy:
\begin{enumerate}
    \item $r_\K\circ\widehat{c}(\alpha-j,j;\incl_{\Sigma\Omega\X})=\widehat{c}(\alpha-j,j;t_\X)\colon\Sigma(\Omega\X)^{\wedge\alpha}\to (C\Omega\X,\Omega\X)^\K$;
    \item $q_\K\circ \widehat{c}(\alpha-j,j;t_\X)=\widehat{c}(\alpha-j,j;\incl_{\Sigma\Omega\X})\colon\Sigma(\Omega\X)^{\wedge\alpha}\to (C\Omega\Sigma\Omega\X,\Omega\Sigma\Omega\X)^\K$.
\end{enumerate}
\end{prp}

\begin{proof}
Since lifts of Whitehead products through the fibration in the bottom row of~(\ref{phpfib}) are unique up to homotopy, statement (1) is equivalent to the identity $t_\K\circ c(\alpha-j,j;\incl_{\Sigma\Omega\X})=c(\alpha-j,j;t_\X)$. This follows by naturality of Whitehead products, since
$t_\K\circ \incl_{\Sigma\Omega\X,i}=t_{\X,i}$.

To prove (2), it is sufficient to show that $\iota\circ q_\K\circ \widehat{c}(\alpha-j,j;t_\X)=c(\alpha-j,j;\incl_{\Sigma\X})$, where $\iota$ is the fibre inclusion in~\eqref{phpfib}.
Consider the diagram
\[
\xymatrix{
\Sigma(\Omega\X)^{\wedge\alpha}
\ar[d]^-{\Sigma(E_{\Omega \X})^{\wedge\alpha}}
\ar[rrr]^-{\widehat{c}(\alpha-j,j;t_{\X})}
&&&
(C\Omega\X,\Omega\X)^\K\ar[d]^-{q_\K}
\\
\Sigma(\Omega\Sigma\Omega\X)^{\wedge\alpha}
\ar[rrr]^-{\widehat{c}(\alpha-j,j;t_{\Sigma\Omega\X})}
\ar[rrrd]_-{c(\alpha-j,j;t_{\Sigma\Omega\X})}
&&&
(C\Omega\Sigma\Omega\X,\Omega\Sigma\Omega\X)^\K
\ar[d]^-{\iota}\\
&&&
(\Sigma\Omega\X,\ast)^\K,
}
\]
which is commutative since $q_\K=(CE_{\Omega\X},E_{\Omega \X})^\K$. 
The composition along the left side of the diagram is an iterated Whitehead product of the maps 
\[
  t_{\Sigma\Omega\X,i}\circ \Sigma E_{\Omega X_i}=
  \incl_{\Sigma\Omega\X,i}\circ\ev_{\Sigma\Omega X_i}
  \circ \Sigma E_{\Omega X_i}=\incl_{\Sigma\Omega\X,i}
  \colon\Sigma\Omega X_i\to (\Sigma\Omega\X,\ast)^\K.
\]
Hence, this composition is $c(\alpha-j,j;\incl_{\Sigma\Omega\X})$,
as required. 
\end{proof}

\begin{thm}
\label{thm:fiber of X^L}
Let $\L$ be the simplicial complex consisting of $m$ disjoint points.
\begin{enumerate}
    \item Let $X_1,\dots,X_m$ be simply connected CW-complexes. Then the map
\[
  \widehat g_\L=\bigvee \widehat{c}(J\setminus j,j;t_\X)\colon
  \bigvee_{J\subset[m]}\bigvee_{j\in J\setminus\max(J)}
  \Sigma(\Omega\X)^{\wedge J} \to (C\Omega\X,\Omega\X)^\L
\]
is a homotopy equivalence. 

\item Let $Y_1,\dots,Y_m$ be connected CW-complexes. Then the map
\[
  \bigvee_{\alpha,j}\widehat{c}(\alpha-j,j;\incl_{\Sigma\Y})
  \colon\bigvee_{\alpha\in\Zm}
  \bigvee_{j\in\supp\alpha\setminus\max(\supp\alpha)}
  \Sigma\Y^{\wedge \alpha} \to (C\Omega\Sigma\Y,\Omega\Sigma\Y)^\L
\]  
is a homotopy equivalence.
\end{enumerate}
\end{thm}

\begin{proof}
(1) is \cite[Theorem~7.2]{theriault_dual} (see \cite[Remark~7.7]{theriault_dual} for the description of the explicit indexing set), and (2) is proved in \cite[Theorem~7.4]{theriault_dual}.
\end{proof}

Here is an analogue of Proposition~\ref{diagDJL} for more general polyhedral products: 

\begin{prp}
\label{prp:f and s for X^L}
For simply connected CW-complexes $X_1,\dots,X_m$ and $\K=\L$, the diagram of homotopy fibrations~\eqref{phpfib} takes the form
\[
  \xymatrix@C=1.5em{
  \bigvee_{\alpha\in\Zm}\bigvee_{j\in \Theta_\L(\alpha)}
  \Sigma(\Omega\X)^{\wedge \alpha}
  \ar[rrr]^-{\vee c(\alpha-j,j;\incl)}
  \ar[d]^{r'_\L}
  &&&
  \bigvee_{i=1}^m \Sigma\Omega X_i
  \ar[r]
  \ar[d] 
  &
  \prod_{i=1}^m\Sigma\Omega X_i\ar[d]\\
  \bigvee_{J\subset[m]}\bigvee_{j\in\Theta_\L(J)}
  \Sigma(\Omega\X)^{\wedge J}
  \ar[rrr]^-{\vee c(J\setminus j,j;t)}
  \ar@<1ex>@/^/[u]^-{q'_{\L}} 
  &&&
  \bigvee_{i=1}^m X_i\ar[r]
  &
  \prod_{i=1}^m X_i,
  }
\]
where $q'_\L$ is the inclusion of a subwedge.
\end{prp}

\begin{proof}
This is a combination of Theorem \ref{thm:fiber of X^L} and Proposition \ref{prp:f and s commute with lifts}.
\end{proof}

The description of the map $r'_\L$ (or map $r_\L$ in~\eqref{phpfib}) in terms of the wedge summands reduces to the following question. 

\begin{prb}
\label{prb:problem-of-identifying-widehat-c}
For $\alpha\in\Zm$ and $j\in\supp\alpha$, describe the image of the map \begin{equation}
\label{eqn:problem-of-identifying-widehat-c}    
  \widehat{c}(\alpha-j,j;t_\X)\colon
  \Sigma(\Omega\X)^{\wedge\alpha}\to (C\Omega\X,\Omega\X)^\L\overset\simeq\longrightarrow
  \bigvee_{J\subset[m]}
  \bigvee_{j\in\Theta_\L(J)}\Sigma(\Omega\X)^{\wedge J}
\end{equation}
under the homotopy equivalence of Theorem \ref{thm:fiber of X^L}.
\end{prb}

\begin{rmk}
For the case $X_i=\CC P^\infty$, the expression of the map $r'_\L$ in terms of the wedge summands is given in Proposition~\ref{diagDJL}. The map~\eqref{eqn:problem-of-identifying-widehat-c} is expressed in terms of Whitehead products and compositions with the iterated Hopf element using Lemma~\ref{lmm:computing c(alpha-j,j)}. For these computations, it is crucial that $\Omega \CC P^\infty\simeq S^1$ is a suspension. There is a similar description in the case $X_i=\HH P^\infty$, with $\eta\in\pi_{n+1}(S^n)$ replaced by the quaternionic Hopf element $\nu\in\pi_{n+3}(S^n)$.

Using a more general distributivity formula \cite[II.(3.5)]{baues}, it might be possible to extend our approach to the case when the spaces $\Omega X_i$ are homotopy equivalent to finite dimensional CW-complexes (e.g. $X_i=BG_i$ for connected Lie groups $G_i$).
\end{rmk}

By Proposition~\ref{prp:f and s for X^L}, the map in~\eqref{eqn:problem-of-identifying-widehat-c} is the inclusion of a wedge summand if $\alpha=\supp\alpha$, i.\,e., if the iterated commutator $c(\alpha-j,j;t_\X)$ has no repeating indices. For the simplest iterated commutator with repeating indices, we have the following description that uses a formula of Baues (see Appendix~\ref{sec:baues_formula}).

\begin{prp}\label{combaues}
Let $i,j\in [m]$ be distinct, and consider the Whitehead product
\[
  [[t_{\X,i},t_{\X,j}],t_{\X,j}]\colon
  \Sigma(\Omega X_i\wedge\Omega X_j\wedge \Omega X_j)\to 
  \bigvee_{i=1}^m X_i
\]  
and its unique lift 
\[
  \widehat{c}\colon
  \Sigma(\Omega X_i\wedge\Omega X_j\wedge \Omega X_j)
  \to(C\Omega\X,\Omega\X)^\L
  \simeq\bigvee_{J,j}\Sigma(\Omega\X)^{\wedge J}.
\] 
Then 
\[
  [[t_{\X,i},t_{\X,j}],t_{\X,j}]\simeq[t_{\X,i},t_{\X,j}]\circ (\id_{\Omega X_i}\wedge H\mu_{X_j}),
\]  
where $H\mu_{X_j}\colon\Sigma(\Omega X_j\wedge\Omega X_j)\to \Sigma\Omega X_j$ is the Hopf construction for the loop space multiplication.
In particular, this identifies the map $\widehat{c}$ with the composition of $\id_{\Omega X_i} \wedge H\mu_{X_j}\colon\Sigma(\Omega X_i\wedge\Omega X_j\wedge\Omega X_j)\to\Sigma(\Omega X_i\wedge\Omega X_j)$ and the inclusion map.
\end{prp}

\begin{proof}
We set $K=\Omega X_i$ and $L=X_j$ in Theorem~\ref{thm:baues_formula} to obtain
\[
  [[i,r],r]=[i,r]\circ(\id_{\Omega X_i}\wedge H\mu_{X_j})\colon
  \Sigma\Omega X_i\wedge\Omega X_i\wedge\Omega X_i\to
  \Sigma\Omega X_i\vee X_j,
\]
where $i\colon\Sigma\Omega X_i\hookrightarrow\Sigma\Omega X_i\vee X_j$ and $r\colon\Sigma\Omega X_j\xrightarrow{\ev_{X_j}} X_j\hookrightarrow\Sigma\Omega X_i\vee X_j$. The required identity is obtained by composing the above with $\ev_{X_i}\!\vee\,\id_{X_j}\colon \Sigma\Omega X_i\vee X_j\to X_i\vee X_j$ and observing that $(\ev_{X_i}\!\vee\,\id_{X_j})\circ i=t_{\X,i}$ and
$(\ev_{X_i}\!\vee\,\id_{X_j})\circ r=t_{\X,j}$.
\end{proof}

Proposition~\ref{combaues} is stated for the left nested commutator $[[t_{\X,i},t_{\X,j}],t_{\X,j}]$, whereas the right nested commutators $[t_i,[t_i,t_j]]$ are used in Proposition~\ref{prp:iij} and throughout the paper. A version of the Baues formula for the right nested commutators would involve a variant of the Hopf construction that is conjugate to the standard one.

\appendix

\section{The groups $A_n$}
\label{sec:groups A_n}
Here we describe the graded abelian groups $A_n\subset\pi_*(S^n)$ generated by the identity $\iota\colon S^n\to S^n$, its compositions with the (suspended) Hopf elements $\eta$, and Whitehead products. 
Note that iterated Whitehead products of $\iota$ with itself of length $\geq 4$ vanish (see \cite[XI, (8.8)]{whitehead}).
Together with the identities from Proposition~\ref{prp:properties of hopf}, this implies that $A_n$ is spanned by the finite set of elements
\[\left\{\iota,\eta^k,[\iota,\iota]\circ\eta^k,[\iota,[\iota,\iota]]\circ\eta^k,\;0\leq k\leq 3\right\}\subset\pi_*(S^n).
\]

\begin{thm}
\label{thm:[i,[i,i]]_in_spheres}
Let $I_n\subset\pi_*(S^n)$ be the quasi-Lie subring generated by $\iota_n\in\pi_n(S^n).$ Then
\[
  I_n=
\begin{cases}
\ZZ\langle\iota_n\rangle&\text{for}\quad n=1,3,7;\\
\ZZ\langle\iota_n\rangle\oplus\ZZ_2\langle[\iota_n,\iota_n]\rangle
&\text{for}\quad n=2k+1,~n\neq 1,3,7;\\
\ZZ\langle\iota_n\rangle\oplus\ZZ\langle[\iota_n,\iota_n]\rangle
&\text{for}\quad n=2;\\
\ZZ\langle\iota_n\rangle\oplus\ZZ\langle[\iota_n,\iota_n]\rangle
\oplus\ZZ_3\langle[\iota_n,[\iota_n,\iota_n]]\rangle
&\text{for}\quad n=2k,~n\neq 2.
\end{cases}
\]
\end{thm}

\begin{proof}
Theorem \ref{thm:free_quasilie_basis} gives the following free quasi-Lie rings generated by~$\iota_n$:
\[
  \FQL(\iota_n)=\begin{cases}
  \ZZ\langle\iota_n\rangle\oplus\ZZ_2\langle[\iota_n,\iota_n]
  \rangle
  &\text{for}\quad n=2k+1;\\
  \ZZ\langle\iota_n\rangle\oplus\ZZ\langle[\iota_n,\iota_n]\rangle
  \oplus\ZZ_3\langle[\iota_n,[\iota_n,\iota_n]]\rangle
  &\text{for}\quad n=2k.
  \end{cases}
\]
Note that here we use the topological grading instead of the algebraic one, so the parity is reversed.
Now the result follows from the following identities in $\pi_*(S^n$):
\begin{enumerate}
\item $[\iota_n,\iota_n]=0$ if and only if $n\in\{1,3,7\};$
\item $[\iota_{2k},[\iota_{2k},\iota_{2k}]]=0$ if and only if $k=1.$
\end{enumerate}
Here (1) follows from Adams' solution of Hopf invariant one problem, and (2) is proved by Liulevicius~\cite[Theorem 6]{liulevicius}.
\end{proof}

\begin{prp}\label{pin+k}
The groups $\pi_{n+k}(S^n)$ for $k=0,\dots,4$ are as follows.
\begin{enumerate}
\item $\pi_n(S^n)=\ZZ\langle\iota_n\rangle$ for $n\geq 1$;

\item $\pi_3(S^2)=\ZZ\langle\eta_2\rangle$ and $\pi_{n+1}(S^n)=\ZZ_2\langle\eta_n\rangle$ for $n\geq3$;

\item $\pi_{n+2}(S^n)=\ZZ_2\langle\eta_n^2\rangle$ for $n\geq 2$;

\item $\pi_5(S^2)=\ZZ_2\langle\eta_2^3\rangle$,
$\pi_6(S^3)=\ZZ_4\langle\nu'\rangle\oplus
\ZZ_3\langle\alpha_1(3)\rangle$,\newline 
$\pi_7(S^4)=\ZZ\langle\nu_4\rangle
\oplus\ZZ_4\langle E\nu'\rangle
\oplus\ZZ_3\langle E\alpha_1(3)\rangle$ and\newline 
$\pi_{n+3}(S^n)=\ZZ_8\langle\nu_n\rangle
\oplus\ZZ_3\langle E^{n-3}\alpha_1(3)\rangle$ 
for $n\geq 5$ \newline
with relations $\eta_3^3=2\nu'$, $\eta_4^3=2E\nu'$ and  $\eta_n^3=4\nu_n$ for $n\geq 5$;

\item 
$\pi_6(S^2)=\ZZ_4\langle\eta_2\circ\nu'\rangle
\oplus\ZZ_3\langle\eta_2\circ\alpha_1(3)\rangle$ and $\pi_7(S^3)=\ZZ_2\langle\nu'\circ\eta_6\rangle$ \newline 
with relations $\eta_2^4=2(\eta_2\circ\nu')$ and $\eta_3^4=0$. 
\end{enumerate}
\end{prp}

\begin{proof}
For the additive description see \cite[Propositions 5.1, 5.3, 5.5, 5.6, 5.8; Theorem 13.4, Proposition 13.6]{toda}. We also have $\eta_3^3=2\nu'$ by \cite[(5.3)]{toda} and $E^2\nu'=2\nu_5$ by \cite[(5.5)]{toda}. It follows that $\eta_4^3=E(\eta_3^3)=2E\nu'$ and $\eta_n^3=E^{n-5}\eta_5^3=4\nu_n$ for $n\geq 5$. Finally, we have $\eta_2^4=\eta_2\circ\eta_3^3=\eta_2\circ (2\nu')=2(\eta_2\circ\nu')$ by Proposition \eqref{prp:composition_properties}(3) and $\eta_3^4=\eta_3^3\circ\eta_6=2\nu'\circ\eta_6=0$ since $\nu'\circ\eta_6$ has order two.
\end{proof}
\begin{crl}
For $n\geq 3,$ the elements $\eta_n,\eta_n^2,\eta_n^3$ are nonzero of order $2$, and $\eta_n^4=0$.\qed
\end{crl}

The following theorem follows from results of Hilton--Whitehead, Kristensen--Madsen, Mahowald, Oshima, Thomeier and Toda. We rely on the exposition by Golasinski and Mukai \cite{golasinski_mukai}.

\begin{thm}
\label{thm:[i,eta^k]} 
The following hold for any $k\geq 0$:
\begin{enumerate}
\item $[\iota_n,\eta_n]=0$ if and only if $n\in\{2,6,4k+3\}$;
\item $[\iota_n,\eta_n^2]=0$ if and only if $n\in\{5,4k+2,4k+3\}$;
\item $[\iota_n,\eta_n^3]=0$ if and only if 
  $n\in\{4,12,4k+1,4k+2,4k+3\}$.
\end{enumerate}
\end{thm}

\begin{proof}
For the first two statements, see \cite[(2.1) and (2.2)]{golasinski_mukai}. We prove (3).

For $n=3$, we have $[\iota_3,\eta_3^3]=0$ since all Whitehead products in $\pi_*(S^3)$ vanish.

For $n=4$, we have 
\[
  [\iota_4,\eta_4^3]=[\iota_4,2E\nu']=2[\iota_4,\iota_4]\circ E^4\nu'=2(2\nu_4-E\nu')\circ 2\nu_7=8\nu_4^2=0,
\]  
where the first identity is by Proposition~\ref{pin+k}~(4), the second by Proposition~\ref{prp:composition_properties}~(6) and the third by~\cite[(5.5),(5.8)]{toda}. The fourth identity uses the fact that  $E\nu'\circ\nu_7=E(\nu'\circ\nu_6)=0$, since $\nu'\circ\nu_6\in\pi_9(S^3)_{(2)}=0$. The last identity follows from $\pi_{10}(S^4)_{(2)}\cong\ZZ_8\langle\nu_4^2\rangle$~\cite[Proposition~5.11]{toda}.

For $n\geq 5$ we have $[\iota_n,\eta_n^3]=4[\iota_n,\nu_n]$ by Proposition~\ref{pin+k}~(4). The orders of elements $[\iota_n,\nu_n]$ are given in \cite[(2.11)]{golasinski_mukai}, whence the result follows. 
\end{proof}

Combining Theorems \ref{thm:[i,[i,i]]_in_spheres} and \ref{thm:[i,eta^k]}, we obtain an additive description of the groups $A_n$.
\begin{thm}
For $n\geq 3$, the graded abelian group $A_n$ has the following structure.
\begin{enumerate}
\item For $n=4k$,
\[
\begin{split}
  A_n=&\ \ZZ\langle\iota\rangle\oplus
  \ZZ_2\langle\eta,\eta^2,\eta^3\rangle\\
  &\oplus
  \ZZ\langle[\iota,\iota]\rangle\oplus
  \ZZ_2\langle[\iota,\iota]\circ\eta,[\iota,\iota]\circ\eta^2,
  [\iota,\iota]\circ\eta^3\rangle\oplus\ZZ_3\langle[\iota,[\iota,\iota]]\rangle,
\end{split}
\]
where $[\iota,\iota]\circ\eta^3=0$ for $n=4,12$;

\smallskip

\item For $n=4k+1$,
\[
  A_n=\ZZ\langle\iota\rangle\oplus
  \ZZ_2\langle\eta,\eta^2,\eta^3\rangle
  \oplus
  \ZZ_2\langle[\iota,\iota],[\iota,\iota]\circ\eta,[\iota,\iota]\circ\eta^2\rangle,\hspace*{0.105\textwidth}
\]
where $[\iota,\iota]\circ\eta^2=0$ for $n=5$;

\smallskip

\item For $n=4k+2$,
\[
  \ \:A_n= \ZZ\langle\iota\rangle\oplus
  \ZZ_2\langle\eta,\eta^2,\eta^3\rangle\oplus
  \ZZ\langle[\iota,\iota]\rangle\oplus
  \ZZ_2\langle[\iota,\iota]\circ\eta\rangle\oplus 
  \ZZ_3\langle[\iota,[\iota,\iota]]\rangle, 
\] 
where $[\iota,\iota]\circ\eta=0$ for $n=6$;

\smallskip

\item For $n=4k+3$,
\[ 
 A_n= \ZZ\langle\iota\rangle\oplus
 \ZZ_2\langle\eta,\eta^2,\eta^3\rangle\oplus 
 \ZZ_2\langle[\iota,\iota]\rangle,\hspace*{0.31\textwidth} 
\]
where $[\iota,\iota]=0$ for $n=3,7$.
\end{enumerate} 
\vspace{-0.8cm}~\qed
\bigskip 
\end{thm}

\section{A formula of Baues}
\label{sec:baues_formula} 
For based CW-complexes $X$ and $Y$, the suspended projection $\Sigma q\colon\Sigma(X\times Y)\to \Sigma X\wedge Y$ has a right homotopy inverse $\sigma\colon\Sigma X\wedge Y\to\Sigma(X\times Y)$, which comes from the homotopy decomposition
\[
  \Sigma(X\times Y)\simeq \Sigma X\vee \Sigma Y\vee (\Sigma X\wedge Y).
\]
Equivalently, the identity map $\id$ on $\Sigma(X\times Y)$ decomposes as 
\begin{equation}\label{idsusp}
  \id\simeq (\Sigma i_X\circ\Sigma \pr_X)
  +(\Sigma i_Y\circ\Sigma \pr_Y)+(\sigma\circ\Sigma q)
\end{equation}
where $i_X\colon X\to X\times Y$ is the inclusion $x\mapsto(x,\pt)$ and similarly for $i_Y\colon Y\to X\times Y$.

The Whitehead product of based maps $f\colon\Sigma X\to Z$ and $g\colon\Sigma Y\to Z$ is the adjoint $[f,g]\colon \Sigma X\wedge Y\to Z$ of the composite $X\wedge Y\xrightarrow{f'\wedge g'}
\Omega Z\wedge\Omega Z\xrightarrow{c}\Omega Z$. Since $\Sigma q\colon\Sigma(X\times Y)\to \Sigma X\wedge Y$ has a right homotopy inverse, the Whitehead product $[\alpha,\beta]\in[\Sigma (X\wedge Y),Z]$ of homotopy classes $\alpha\in[\Sigma X,Z]$ and $\beta\in[\Sigma X,Z]$ is uniquely determined by the identity
\begin{equation}\label{Whcom}
  [\alpha,\beta]\circ\Sigma q=
  \lbr\alpha\circ\Sigma\pr_X,\beta\circ\Sigma\pr_Y\rbr\in [\Sigma(X\times Y), Z].
\end{equation}
Here $\pr_X\colon X\times Y\to X$ and $\pr_Y\colon X\times Y\to Y$ are the projections, and 
\[
  \lbr\varphi,\psi\rbr:=-\varphi-\psi+\varphi+\psi
\]  
is the commutator of elements $\varphi$, $\psi$ in the group $[\Sigma (X\times Y),Z]$. It is more convenient to use additive notation for the group operation, although the group $[\Sigma (X\times Y),Z]$ is noncommutative.

The \emph{Hopf invariant} $Hf\in [\Sigma X\wedge Y,\Sigma Z]$ of a based map $f\colon X\times Y\to Z$ is defined as follows. Let $f_X\colon X\to Z$ and $f_Y\colon Y\to Z$ be the restrictions of~$f$. Then $Hf$ is uniquely determined by the equation
\begin{equation}\label{Hopf}
  Hf\circ\Sigma q=-\Sigma f_X\circ\Sigma\pr_X
  -\Sigma f_Y\circ\Sigma\pr_Y+\Sigma f\in [\Sigma(X\times Y), \Sigma Z].
\end{equation}

As a special case, let $\mu\colon \Omega L\times\Omega L\to\Omega L$ be the loop multiplication, and $\pr_{\Omega L,1}$, $\pr_{\Omega L,2}\colon \Omega L\times\Omega L\to\Omega L$ be the projections. Then $H\mu\in[\Sigma(\Omega L\wedge\Omega L),\Sigma\Omega L]$ is uniquely determined by the identity
\[
  H\mu\circ\Sigma q=-\Sigma\pr_{\Omega L,1}-\Sigma \pr_{\Omega L,2}
  +\Sigma\mu\in[\Sigma(\Omega L\times\Omega L),\Sigma\Omega L].
\]  
Composing with $\sigma$ and noting that $\Sigma\pr_{\Omega L,1}\circ\sigma\simeq
\Sigma\pr_{\Omega L,2}\circ\sigma\simeq\ast$, we obtain
\[
  H\mu=\Sigma\mu\circ\sigma\colon
  \Sigma\Omega L\wedge\Omega L\stackrel{\sigma}{\longrightarrow} 
  \Sigma(\Omega L\times\Omega L)\xrightarrow{\Sigma\mu}{\Sigma\Omega L}.
\]

\begin{lmm}
\label{lmm:ev distributes multiplication}
Let $\ev_L\colon\Sigma\Omega L\to L$ be the evaluation map. Then
\[
  \ev_L\circ\Sigma\mu = \ev_L\circ \Sigma\pr_{\Omega L,1}+ \ev_L\circ\Sigma\pr_{\Omega L,2}
  \;\in [\Sigma(\Omega L\times\Omega L),L].
\]  
\end{lmm}

\begin{proof} 
To start, compose $\Sigma\mu$ with the decomposition~\eqref{idsusp} of the identity map on $\Sigma(\Omega L\times\Omega L)$ to obtain
\begin{align*} 
  \Sigma\mu\simeq\Sigma\mu\circ \id &  
  \simeq \Sigma\mu\circ\bigl(
  (\Sigma i_{\Omega L,1}\circ \Sigma\pr_{\Omega L,1})
  +(\Sigma i_{\Omega L,2}\circ\Sigma\pr_{\Omega L,2})
  +(\sigma\circ\Sigma q)\bigr) \\ 
   & \simeq\Sigma\pr_{\Omega L,1}
   +\Sigma\pr_{\Omega L,2}+(\Sigma\mu\circ\sigma\circ\Sigma q). 
\end{align*} 
where the last identity uses the fact that $\mu\circ i_{\Omega L,1}\simeq\id$ and $\mu\circ i_{\Omega L,1}\simeq\id$.

Ganea~\cite{ganea} proved that there is a homotopy fibration 
\[
  \Sigma\Omega L\wedge\Omega L\xrightarrow{H\mu} 
  \Sigma\Omega L\xrightarrow{\ev_{L}}  L.
\] 
It follows that $\ev_{L}\circ\Sigma\mu\circ\sigma$ is null homotopic, 
since $\ev_{L}$ and $H \mu=\Sigma\mu\circ\sigma$ are consecutive 
maps in a homotopy fibration. Thus, 
\begin{align*} 
  \ev_{L}\circ\Sigma\mu  
  &\simeq \ev_{L}\circ(\Sigma\pr_{\Omega L,1}+\Sigma\pr_{\Omega L,2}+(\Sigma\mu\circ\sigma\circ\Sigma q))\\  
  &\simeq\ev_{L}\circ\Sigma\pr_{\Omega L,1}+
  \ev_{L}\circ\Sigma \pr_{\Omega L,2}.\qedhere  
\end{align*}
\end{proof}

The following formula is given in \cite{baues_homotopy} with a note that its proof is similar to \cite[(3.1.22)]{baues_obstruction}.
We provide a full proof below.

\begin{thm}[{\cite[(A.1.23)]{baues_homotopy}}]
\label{thm:baues_formula}
For based CW-complexes $K$ and $L$, let $Z=\Sigma K\vee L$, let $i\colon\Sigma K\hookrightarrow \Sigma K\vee L$ and $r\colon\Sigma\Omega L\xrightarrow{\ev_L} L\hookrightarrow \Sigma K\vee L$. Then 
\[
  [[i,r],r]=[i,r]\circ (\id_K\wedge H\mu)\;
  \in[\Sigma K\wedge\Omega L\wedge \Omega L,Z].
\]
More precisely, the following diagram is homotopy commutative:
\[
  \xymatrix{
  \Sigma((K\wedge\Omega L)\wedge\Omega L)
  \ar[d]^-\simeq
  \ar[rr]^-{[[i,r],r]}
  &&
  Z\\
  K\wedge \Sigma(\Omega L\wedge\Omega L)
  \ar[r]^-{\id_K\wedge H\mu}
  &
  K\wedge\Sigma\Omega L
  \ar[r]^-\simeq
  &
  \Sigma(K\wedge\Omega L).
  \ar[u]_{[i,r]}
}
\]
\end{thm}

\begin{proof}
Since $\Sigma q\colon \Sigma(K\times\Omega L\times\Omega L)\to
\Sigma K\wedge\Omega L\wedge\Omega L$ has a right inverse, we can instead prove the identity obtained by composing with $\Sigma q$. Composing the left hand side and using~\eqref{Whcom} gives
\begin{equation}\label{irr}
  [[i,r],r]\circ\Sigma q=\lbr\lbr i\circ\Sigma\pr_K,
  r\circ\Sigma\pr_{\Omega L,1} \rbr,
  r\circ\Sigma\pr_{\Omega L,2} \rbr.
\end{equation}
The composition of the right hand side with $\Sigma q$ is described by the homotopy commutative diagram, which follows from~\eqref{Whcom} and~\eqref{Hopf}: 
\[
  \xymatrix{
  \Sigma(K\times\Omega L\times\Omega L)
  \ar[d]^-{\Sigma q}
  \ar[rrrr]^-{\Sigma(\id_K\times(-\pr_{\Omega L,1}-
  \pr_{\Omega L,2}+\mu))}
  &&&&
  \Sigma(K\times\Omega L)
  \ar[d]^-{\Sigma q}
  \ar[rrr]^-{\lbr i\circ\Sigma\pr_K,
  r\circ\Sigma\pr_{\Omega L}\rbr}
  &&&
  Z %\Sigma K\vee L
  \ar@{=}[d]\\
  \Sigma K\wedge \Omega L\wedge\Omega L
  \ar[rrrr]^-{\id_K\wedge H\mu}
  &&&&
  \Sigma K\wedge\Omega L
  \ar[rrr]^{[i,r]}
  &&&
  Z. %\Sigma K\vee L
}
\]
We therefore have
\begin{align*}
  [i,r]&\circ (\id_K\wedge H\mu)\circ\Sigma q =
  \lbr i\circ\Sigma\pr_K,
  r\circ\Sigma\pr_{\Omega L}\rbr \\
  &\ \ \ \circ (-\Sigma(\id_K\times\pr_{\Omega L,1})
  -\Sigma(\id_K\times\pr_{\Omega L,2})
  +\Sigma(\id_K\times\mu) )\\
  &=\lbr r\circ\Sigma\pr_{\Omega L}
   \circ \Sigma(\id_K\times\pr_{\Omega L,1}),
   i\circ \Sigma\pr_K\circ\Sigma(\id_K\times\pr_{\Omega L,1})\rbr\\
  &\ \ \ +\lbr r\circ\Sigma\pr_{\Omega L}
   \circ \Sigma(\id_K\times\pr_{\Omega L,2}),
   i\circ \Sigma\pr_K\circ\Sigma(\id_K\times\pr_{\Omega L,2})\rbr\\
  &\ \ \ +\lbr i\circ \Sigma\pr_K \circ\Sigma(\id_K\times\mu),
  r\circ \Sigma\pr_{\Omega L} \circ\Sigma(\id_K\times \mu)\rbr\\
  &=\lbr r\circ \Sigma\pr_{\Omega L,1},i\circ \Sigma\pr_K\rbr
   +\lbr r\circ \Sigma\pr_{\Omega L,2},i\circ \Sigma\pr_K\rbr\\
  &\ \ \ +\lbr i\circ \Sigma\pr_K, r\circ \Sigma\mu\circ 
   \Sigma\pr_{\Omega L\times\Omega L}\rbr\\
   &=\lbr r\circ \Sigma\pr_{\Omega L,1},i\circ \Sigma\pr_K\rbr
   +\lbr r\circ \Sigma\pr_{\Omega L,2},i\circ \Sigma\pr_K\rbr\\
  &\ \ \ +\lbr i\circ \Sigma\pr_K, (r\circ \Sigma\pr_{\Omega L,1}+ r\circ\Sigma\pr_{\Omega L,2})
   \circ \Sigma\pr_{\Omega L\times\Omega L}\rbr\\
  &=\lbr r\circ \Sigma\pr_{\Omega L,1},i\circ \Sigma\pr_K\rbr
   +\lbr r\circ \Sigma\pr_{\Omega L,2},i\circ \Sigma\pr_K\rbr\\
  &\ \ \  +\lbr i\circ \Sigma\pr_K, r\circ \Sigma\pr_{\Omega L,1}+ r\circ\Sigma\pr_{\Omega L,2}\rbr\\
  &=\lbr\lbr i\circ\Sigma\pr_K,
   r\circ\Sigma\pr_{\Omega L,1} \rbr,
   r\circ\Sigma\pr_{\Omega L,2} \rbr.
\end{align*}
Here the second identity uses the distributivity of the commutator with respect to suspended maps, the third and fifth identities follow by composing the projections, the fourth identity is by Lemma~\ref{lmm:ev distributes multiplication} and the definition of~$r$, and the last is the Witt--Hall identity 
$\lbr b,a\rbr+\lbr c,a\rbr+\lbr a,b+c\rbr=\lbr \lbr a,b\rbr,c\rbr$. The required formula follows by comparing the identity above with~\eqref{irr}.
\end{proof}

\end{document}